\RequirePackage[l2tabu,orthodox]{nag}
\documentclass[submission,copyright,creativecomons,noderivs,10pt]{eptcs}
\providecommand{\event}{QPL 2019}

\pdfoutput=1

\usepackage[suppress]{rwd-drafting} 

\usepackage[T1]{fontenc}
\usepackage{lmodern}
\usepackage{fixltx2e}
\usepackage{microtype} 
\usepackage{xspace}
\usepackage{enumitem}

\makeatletter
\newcommand\etc{etc\@ifnextchar.{}{.\@}\xspace}
\newcommand\ie{i.e.\@\xspace}  % these two may be broken

\makeatother

\newcommand{\suck}{\vspace{-1.0em}}
\newcommand{\ssuck}{\vspace{-1.5em}}
\newcommand{\sssuck}{\vspace{-2.0em}}

\usepackage{graphicx}

\newcommand{\inlinegraphic}[2]{
  \dimendef\grafheight=255\dimendef\grafvshift=254
  \grafheight=#1
  \grafvshift=-0.5\grafheight
  \advance\grafvshift by 0.5ex
  \raisebox{\grafvshift}{\includegraphics[height=\grafheight]{images/#2}\xspace}
}

\newcommand{\ninlinegraphic}[2][1.0]{
  \dimendef\grafheight=255\dimendef\grafvshift=254
  \setbox0 = \hbox{\scalebox{#1}{\includegraphics{images/#2}}}
  \grafheight=\the\ht0
  \grafvshift=-0.5\grafheight
  \advance\grafvshift by 0.5ex
  \raisebox{\grafvshift}{\includegraphics[height=\grafheight]{images/#2}\xspace}
}

\usepackage{latexsym}
\usepackage{amssymb}
\usepackage{amsmath} % can collide with some journals
\usepackage{stmaryrd}
\usepackage{mathtools}
\usepackage{amsthm}
\newtheorem{theorem}{Theorem}[section]
\newtheorem{proposition}[theorem]{Proposition}
\newtheorem{lemma}[theorem]{Lemma}
\newtheorem{corollary}[theorem]{Corollary}
\theoremstyle{definition}\newtheorem{example}[theorem]{Example}
\theoremstyle{definition}
\theoremstyle{definition}\newtheorem{definition}[theorem]{Definition}
\theoremstyle{definition}
\theoremstyle{definition}\newtheorem{remark}[theorem]{Remark}
\theoremstyle{definition}

\newtheorem*{note}{Note}

\newenvironment{repthm}[3] 
{ 
  \bgroup 
  \addtocounter{#2}{-1} 
  \expandafter\def\csname the#2\endcsname{\ref{#3}} 
  \def\foo{\end{#1}} 
  \begin{#1} 
} 
{ 
  \foo 
  \egroup 
}

\newenvironment{prevtheorem}[1]
{\begin{repthm}{theorem}{theorem}{#1}}
{\end{repthm}}

\newenvironment{prevlemma}[1]
{\begin{repthm}{lemma}{theorem}{#1}}
{\end{repthm}}

\newcommand{\isomorphism}{\cong}
\ifx\numericids\undefined
\newcommand{\id}[1]{\ensuremath{\mathrm{id}_{#1}}}
\else
\newcommand{\id}[1]{\ensuremath{1_{#1}}}
\fi

\newcommand{\op}{^\mathrm{op}}

\newcommand{\catC}{\ensuremath{{\cal C}}\xspace}

\newcommand{\catRel}{% the category of sets and relations
\ensuremath{\mathbf{Rel}}\xspace}

\newcommand{\zxcalculus}{\textsc{zx}-calculus\xspace}

\usepackage{dsfont}

\newcommand{\Hop}{\ensuremath{H^{\mathrm{op}}\xspace}}

\usepackage{hyperref}

\usepackage{tikzit}  
\newcommand{\inlinetikzfig}[2][1.0]{
  \dimendef\grafheight=255\dimendef\grafvshift=254
  \setbox0 = \hbox{\scalebox{#1}{\input{#2.tikz}}}
  \grafheight=\the\ht0
  \grafvshift=-0.5\grafheight
  \advance\grafvshift by 0.5ex
  \raisebox{\grafvshift}{\input{#2.tikz}}
}

\newcommand{\inltf}[1]{\inlinetikzfig{#1}}

\pgfdeclarelayer{background} 
\pgfdeclarelayer{foreground}
\pgfdeclarelayer{edgelayer}
\pgfdeclarelayer{nodelayer}
\pgfsetlayers{background,edgelayer,nodelayer,main,foreground} 

\tikzstyle{halfsize}=[x=0.5cm, y=0.5cm]
\tikzstyle{normalsize}=[]
\tikzstyle{doublesize}=[]

\tikzstyle{(null)}=[]
\tikzstyle{plain}=[]

\usepackage{tikzzx}
\usepackage{tikzcircuits}
\usepackage[undirected]{tikzquanto} 

\newcommand{\daggerantipode}{\begin{tikzpicture}
	\begin{pgfonlayer}{nodelayer}
		\node [style=green smallvertex] (0) at (0, 0) {};
		\node [style=red smallvertex] (1) at (0.25, 0.25) {};
		\node [style=none] (2) at (-0.25, 0.25) {};
		\node [style=none] (3) at (0.5, 0) {};
	\end{pgfonlayer}
	\begin{pgfonlayer}{edgelayer}
		\draw [bend right=15] (2.center) to (0);
		\draw [in=-135, out=45] (0) to (1);
		\draw [bend right=15] (3.center) to (1);
	\end{pgfonlayer}
\end{tikzpicture}}

\newcommand{\antipodeform}{\begin{tikzpicture}
	\begin{pgfonlayer}{nodelayer}
		\node [style=green smallvertex] (0) at (-0.25, 0.25) {};
		\node [style=red smallvertex] (1) at (0, 0) {};
		\node [style=none] (2) at (-0.5, 0) {};
		\node [style=none] (3) at (0.25, 0.25) {};
	\end{pgfonlayer}
	\begin{pgfonlayer}{edgelayer}
		\draw [bend left=15] (2.center) to (0);
		\draw [in=135, out=-45] (0) to (1);
		\draw [bend left=15] (3.center) to (1);
	\end{pgfonlayer}
\end{tikzpicture}}
\newcommand{\antipodeiform}{\begin{tikzpicture}
	\begin{pgfonlayer}{nodelayer}
		\node [style=green smallvertex] (0) at (0, 0) {};
		\node [style=red smallvertex] (1) at (-0.25, 0.25) {};
		\node [style=none] (2) at (0.25, 0.25) {};
		\node [style=none] (3) at (-0.5, 0) {};
	\end{pgfonlayer}
	\begin{pgfonlayer}{edgelayer}
		\draw [bend left=15] (2.center) to (0);
		\draw [in=-45, out=135] (0) to (1);
		\draw [bend left=15] (3.center) to (1);
	\end{pgfonlayer}
\end{tikzpicture}}

\newcommand{\rint}{\begin{tikzpicture}
	\begin{pgfonlayer}{nodelayer}
		\node [style=integral] (0) at (0, 0.25) {};
		\node [style=none] (1) at (0, 0) {};
	\end{pgfonlayer}
	\begin{pgfonlayer}{edgelayer}
		\draw (0) to (1.center);
	\end{pgfonlayer}
\end{tikzpicture}
}
\newcommand{\gcoint}{\begin{tikzpicture}
	\begin{pgfonlayer}{nodelayer}
		\node [style=cointegral] (0) at (0, 0) {};
		\node [style=none] (1) at (0, 0.25) {};
	\end{pgfonlayer}
	\begin{pgfonlayer}{edgelayer}
		\draw (1.center) to (0);
	\end{pgfonlayer}
\end{tikzpicture}
}

\newcommand{\antipode}{
\begin{tikzpicture}
  \node [style=antipode] at (0, 0) {};
\end{tikzpicture}
}

\newcommand{\antipoded}{
\begin{tikzpicture}[rotate =90, transform shape]
  \node [style=antipode] at (0, 0) {};
\end{tikzpicture}
}

\newcommand{\rediso}{
\begin{tikzpicture}
  \node [style=red-iso] at (0, 0) {};
\end{tikzpicture}
}

\newcommand{\redinv}{
\begin{tikzpicture}
  \begin{pgfonlayer}{nodelayer}
    \node [style=red-iso] (24) at (1, 0) {};
    \node [style=none] (25) at (1, 0) {\scriptsize -1};
  \end{pgfonlayer}
\end{tikzpicture}
}

\tikzset{every picture/.style={line width=0.75pt}} 
\tikzset{dedge/.style={postaction={decorate,decoration={markings,mark=at position 0.5 with {\arrow{>}}}}}}

\usepackage{mathrsfs, float}

\newcommand{\hopfisfrob}[1]{\inltf{hopf-is-frob/#1}}
\newcommand{\inlfropf}[1]{\inltf{hopf-is-frob/#1}}

\newcommand{\OX}{\ensuremath{\otimes}\xspace}

\newcommand{\FVect}{\ensuremath{\mathbf{FVect_k}}\xspace}

\newcommand{\FPMod}{\ensuremath{\mathbf{FPMod_R}}\xspace}
\newcommand{\Isec}{\ensuremath{\mathcal{I}}\xspace}

\tikzstyle{directed edge}=[-]

\input{front-matter.tex}
\begin{document}

% For article.cls abstract goes here 
% other classes may have it in front-matter
\maketitle 
\begin{abstract}
  The \zxcalculus and related theories are based on so-called
  interacting Frobenius algebras, where a pair of $\dag$-special
  commutative Frobenius algebras jointly form a pair of Hopf algebras.
  In this setting we introduce a generalisation of this structure,
  \emph{Hopf-Frobenius algebras}, starting from a single Hopf algebra
  which is not necessarily commutative or cocommutative.  We provide a
  few necessary and sufficient conditions for a Hopf algebra to be a
  Hopf-Frobenius algebra, and show that every Hopf algebra in \FVect
  is a Hopf-Frobenius algebra. In addition, we show that this
  construction is unique up to an invertible scalar. Due to this fact,
  Hopf-Frobenius algebras provide two canonical notions of duality,
  and give us a ``dual'' Hopf algebra that is isomorphic to the usual
  dual Hopf algebra in a compact closed category. We use this
  isomorphism to construct a Hopf algebra isomorphic to the Drinfeld
  double, but has a much simpler presentation.
\end{abstract}

% We define and study two structures on Hopf algebras on a symmetric
% monoidal category, induced by integrals and cointegrals, which each
% provide us with a dual Hopf algebra. We begin by looking at integral
% Hopf algebras, and show that this gives us an invertible antipode and a
% dual Hopf algebra. We then expand on this concept and look at a
% Hopf-Frobenius algebra, which is a structure where we have two distinct
% Hopf algebras such that they interact  to form two Frobenius algebras.

\section{Introduction}
\label{sec:introduction}

In the monoidal categories approach to quantum theory
\cite{AbrCoe:CatSemQuant:2004,Coecke2017Picturing-Quant} Hopf algebras
\cite{Sweedler1969Hopf-Algebras} have a central role in the
formulation of complementary observables \cite{Coecke:2009aa}.  In
this setting, a quantum observable is represented as special
commutative \dag-Frobenius algebra; a pair of such observables are
called \emph{strongly complementary} if the algebra part of the first
and the coalgebra part of the second jointly form a Hopf algebra.  In
abstract form, this combination of structures has been studied under
the name ``interacting Frobenius algebras''
\cite{Duncan2016Interacting-Fro} where it is shown that relatively
weak commutation rules between the two Frobenius algebras produce the
Hopf algebra structure.  From a different starting point Bonchi et al
\cite{Bonchi2014aJournal} showed that a distributive law between two
Hopf algebras yields a pair of Frobenius structures, an approach which
has been generalised to provide a model of Petri nets
\cite{Bonchi:2019:DAL:3302515.3290338}. Given the similarity of the
two structures it is appropriate to consider both as exemplars of a
common family of \emph{Hopf-Frobenius algebras}.

In the above settings, the algebras considered are both commutative
and cocommutative.  However more general Hopf algebras, perhaps not
even symmetric, are a ubiquitous structure in mathematical physics,
finding applications in gauge theory \cite{Meusburger2017},
topological quantum field theory \cite{balsam2012kitaev} and
topological quantum computing \cite{buerschaper2013hierarchy}.  In
this paper we take the first steps towards generalising the concept of
Hopf-Frobenius algebra to the non-commutative case, and opening the
door to applications of categorical quantum theory in other areas of
physics.

Loosely speaking, a Hopf-Frobenius algebra consists of two monoids and
two comonoids such that one way of pairing a monoid with a comonoid
gives two Frobenius algebras, and the other pairing yields two Hopf
algebras, with the additional condition that antipodes are constructed
from the Frobenius forms.  This schema is illustrated in
Figure~\ref{fig:hopf-frob-alg}.  In
Section~\ref{sec:when-hopf-algebras} we give the precise definition of
Hopf-Frobenius algebras and state the necessary and sufficient
conditions to extend a Hopf algebra to a Hopf-Frobenius algebra in an
arbitrary symmetric monoidal category.  It was previously known that
in \FVect, the category of finite dimensional vector spaces, every
Hopf algebra carries a Frobenius algebra on both its monoid
\cite{Larson1969An-Associative-} and its comonoid
\cite{Doi2000Bi-Frobenius-al,KOPPINEN1996256}; in fact every Hopf
algebra in \FVect is Hopf-Frobenius.  In Section~\ref{sec:examples} we
briefly present some examples which are not the usual abelian group
algebras.  In Section~\ref{sec:drinfeld-double} we show the structure
of a Hopf-Frobenius algebra can be used to give a simpler version of
the Drinfeld double construction.

\paragraph*{Acknowledgements}
The authors wish to thank Dr Gabriella B\"ohm (Wigner Research Centre
for Physics) for her very kind email, and all of the help and input
that she gave us. The authors also wish to thank the anonymous
reviewer for their many useful and insightful comments. Joseph Collins
is supported by a Carnegie Trust PhD Scholarship.

 \begin{figure}[t!]
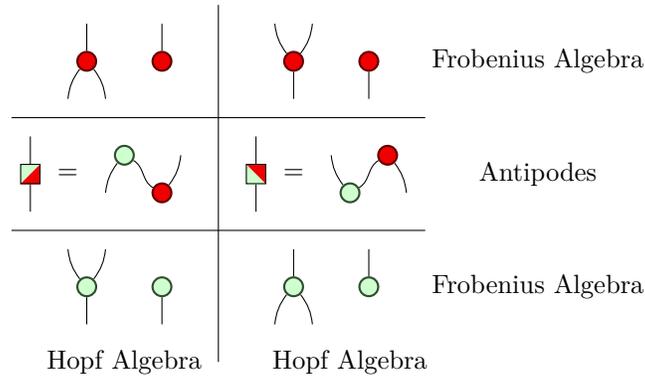

    \centering
    \[
    \hopfisfrob{the-lattice}
    \]
    \suck
    \caption{The elements of a Hopf-Frobenius algebra}    
    \label{fig:hopf-frob-alg}
    \suck
\end{figure}

\section{Preliminaries}
\label{sec:preliminaries}
%%% these labels are the former subsections.  To be restored to their
%%% proper place in the long version.
\label{sec:categ-with-duals}
\label{sec:frobenius-algebras}
\label{sec:hopf-algebras}

We assume that the reader is familiar with strict symmetric monoidal
categories and their diagrammatic notation; see Selinger
\cite{Selinger:2009aa} for a thorough treatment.  We make the
convention that diagrams are read from top to bottom.  When we work
with the dual of an object, we will opt to omit the object names from
the wires except where doing so would create ambiguity.  Instead, we
will assign an orientation to the wires: downwards for the original
object, upwards for its dual.

%% Since we work in categories with duals, the wires carry an
%% orientation: downwards for the original object, upwards for its
%% dual.  We will otherwise omit the object names from the wires
%% except where doing so would create ambiguity.

% removing all the subsections for space reasons
% \subsection{Categories with duals}
% \label{sec:categ-with-duals}

\begin{definition}\label{def:duals}
  In a monoidal category \catC with objects $A$ and $B$, $B$ is
  \emph{left dual} to $A$ if there exists morphisms $d:I\to A\otimes B$
  and $e: B \otimes A \to I$ such that
  \[
  \inlinetikzfig{left-snake} = \ \inltf{idA}
  \qquad \text{ and } \qquad 
  \inltf{right-snake} = \ \inltf{idB}
  \]
  In this circumstance $A$ is \emph{right dual} to $B$.  Note that if
  \catC is symmetric then left duals and right duals coincide.
\end{definition}

The morphisms $d$ and $e$ are usually called the unit and counit; for
reasons which will become obvious shortly we avoid that terminology
and refer to them as the \emph{cap} and the \emph{cup}.  Note that if
an object has a dual it is unique up to isomorphism (see
Lemma~\ref{prop:duals-are-unique}). 

\begin{definition}\label{def:comclcat}
  A \emph{compact closed category} \cite{KelLap:comcl:1980} is a
  symmetric monoidal category where every object $A$ has an assigned
  dual $(A^*,d_A,e_A)$.  In the graphical notation we depict the cup
  and cap in the obvious way: \suck
  \[
  d_A := \ \inltf{cap} \qquad\qquad e_A := \ \inltf{cup}
  \]
\end{definition}

\begin{proposition}[\cite{KelLap:comcl:1980}] \label{prop:dual-functor}
  Let \catC be a compact closed category.  By defining $f^* : B^* \to
  A^*$ as
  \[
  \inltf{f-dual-map} \ := \ 
  \inltf{f-dual-defn}
  \]
  the assignment of duals $A \mapsto A^*$ extends uniquely to a strong
  monoidal functor $(\cdot)^* : \catC\op \to \catC$, with natural
  isomorphisms $(A \otimes B)^* \isomorphism B^* \otimes A^*$, $A^{**}
  \isomorphism A$, and $I^* \isomorphism I$ and, further, $d$ and $e$
  are natural transformations.
\end{proposition}

%% Since we have the freedom to choose, we will take the isomorphisms
%% listed above as the definition of the dual.  This means that the cap
%% for tensor products of objects is ``nested'', as shown below,
%% \suck
%% \[
%% \inltf{tensor-cap} \ := \ \inltf{nested-cap} \;,
%% \]
%% and similarly for the cup.

\begin{remark}
  Note that if $A$ is its own dual, a further collection of coherence
  equations must apply;  see Selinger \cite{Selinger:2010:aa}.
\end{remark}

% \subsection{Frobenius algebras}
% \label{sec:frobenius-algebras}

The main foci of this work -- Frobenius and Hopf algebras -- combine
the structure of a monoid and a comonoid on the same
object.  See Appendix~\ref{sec:monoids-comonoids} for basic definitions.

\begin{definition}\label{def:frobenius-alg}
  A \emph{Frobenius algebra} in a symmetric monoidal category \catC
  consists of a monoid and a comonoid on the same object, obeying the
  Frobenius law, shown below on the left:
  \[
  \inltf{green-frob-law-left} = 
  \inltf{green-frob-law-middle} =
  \inltf{green-frob-law-right}
  \qquad\qquad \qquad 
  \inltf{green-frob-special} = \inltf{identity_1}
  \]
  A Frobenius algebra is called \emph{special} or \emph{separable}
  when it obeys the equation above right, and \emph{quasi-special}
  when it obeys the special equation up to an invertible scalar
  factor.  A Frobenius algebra is commutative when its monoid is, and
  cocommutative when its comonoid is.
\end{definition}

\begin{lemma}\label{lem:frob-self-dual}
  Every Frobenius algebra induces a cup and a cap which make the
  object self-dual.
  \begin{proof}
    Given the Frobenius algebra $(\gmult,\gunit,\gcomult,\gcounit)$
    define the cup and cap as shown below.
    \[
    d := \inltf{green-cap}  = \inltf{green-cap-def}
    \qquad    \qquad    \qquad
    e := \inltf{green-cup}  = \inltf{green-cup-def}
    \]
    From here the snake equation follows directly.
  \end{proof}
\end{lemma}

Definition \ref{def:frobenius-alg}, due to Carboni and Walters
\cite{Carbonia:1987aa}, has a pleasing symmetry between the monoid and
comonoid parts.  However, an older equivalent definition will be
useful in later sections\footnote{See Fauser's survey
\cite{Bertfried:2013aa} for several equivalent definitions.}.

\begin{definition}\label{def:frob-alt}
  A \emph{Frobenius algebra} in a symmetric monoidal category \catC
  consists of a monoid $(F,\gmult,\gunit)$ and a \emph{Frobenius form}
  $\beta : F \otimes F \to I$, which admits an inverse,
  $\bar\beta:I\to F\otimes F$, satisfying:
  \suck
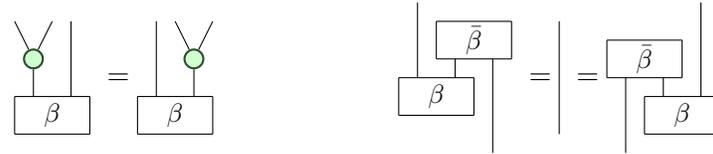
  \[
  \inltf{green-alt-frob-assoc-left} = 
  \inltf{green-alt-frob-assoc-right}
  \qquad  \qquad  \qquad
  \inltf{alt-frob-nondegen-left} =
  \inltf{identity_1} = 
  \inltf{alt-frob-nondegen-right}
  \]
  \suck
\end{definition}
To see that Definition~\ref{def:frobenius-alg} implies this definition
it suffices to take the cup and cap defined above as $\beta$ and
$\bar\beta$.  For the converse, we dualise \gmult with $\beta$ to get
a comonoid. For a proof of how this comonoid fulfills the Frobenius
law, see Kock \cite{Kock:TQFT:2003}

Frobenius forms are far from unique: there is one for each invertible
element of the monoid (see Appendix \ref{sec:more-frobenius-algebras}).

% \subsection{Hopf algebras}
% \label{sec:hopf-algebras}

Special Frobenius algebras can be understood as arising from a
distribution law of comonoids over monoids \cite{Lack:2004sf}.  In the
other direction, distributing monoids over comonoids yields
bialgebras.

\noindent
\begin{note}
  Unlike the preceding section, in our discussion of bialgebras and
  Hopf algebras, we will use different colours for the monoid and
  comonoid parts of the structure.
\end{note}

\begin{definition}\label{def:bialgebra}
  A \emph{bialgebra} in symmetric monoidal category \catC consists of
  a monoid and a comonoid on the same object, which jointly obey the
  \emph{copy}, \emph{cocopy}, \emph{bialgebra}, and \emph{scalar} laws
  depicted below.
  \ssuck
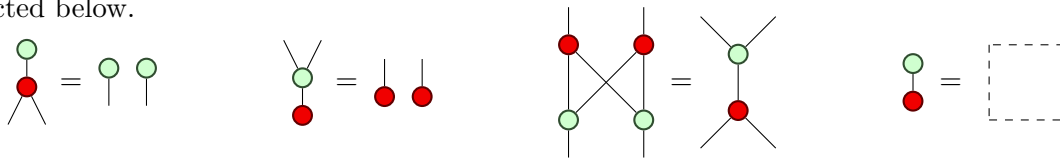
  \[
  \inltf{copy-lhs} \! = \ \inltf{copy-rhs}  
  \qquad  \qquad 
  \inltf{cocopy-lhs} \! = \ \inltf{cocopy-rhs}  
  \qquad  \qquad 
  \inltf{bialg-lhs} =  \inltf{bialg-rhs}
  \qquad\qquad
  \inltf{bialg-bone} = \ \ \inltf{empty}
  \]
  \suck

  \noindent
  Note that the dashed box above represents an empty diagram.  We may
  equivalently define a bialgebra as a monoid and a comonoid such that
  the comonoid is a monoid homomorphism.  A \emph{bialgebra morphism}
  is an morphism of the object which is both a monoid homomorphism and
  a comonoid homomorphism.
\end{definition}

\begin{remark}
  Some works, notably on the \zxcalculus
  \cite{Coecke:2009aa,Backens:2015aa,Jeandel2017A-Complete-Axio} and
  related theories \cite{Duncan2016Interacting-Fro}, the last axiom is
  dropped and the other equations modified by a scalar factor, to give
  a \emph{scaled bialgebra}.  Here we use the standard definition:
  the Frobenius algebras we construct will not be special.
\end{remark}

%Look at review remark 11, we said that we may change it in the long run.

\begin{definition}\label{def:hopf-alg}
  A \emph{Hopf algebra} consists of a bialgebra
  $(H,\gmult,\gunit,\rcomult,\rcounit)$ and an endomorphism $s:H\to H$
  called the \emph{antipode} which satisfies the \emph{Hopf law}:
  \[
  s \ := \  \inltf{antipode} 
  \qquad  \qquad  \qquad
  \inltf{hopf-lhs-ii} \!= \ \inltf{hopf-rhs}  = \  \inltf{hopf-lhs} 
  \]
\end{definition}

Where unambiguous, we abuse notation slightly and use $H$ to refer the
whole Hopf algebra.  Following Street \cite{street2007quantum}, we can
define another Hopf algebra $\Hop\,$ on the same object,
having the same unit and counit, but with the arguments of the
multiplication and comultiplication swapped:
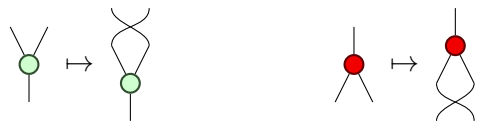
\[
\inltf{green-mult} \mapsto \ \inltf{green-swap-mult}
\qquad\qquad\qquad
\inltf{red-comult} \mapsto \ \inltf{red-comult-swap}
\]
Replacing only the comultiplication as above yields a bialgebra $H^\sigma$ which
is not necessarily Hopf.  We quote the following basic properties from
Street \cite{street2007quantum}.
\begin{proposition}\label{prop:antipode-props} For a Hopf algebra $H$:
  \begin{enumerate}
  \item The antipode $s$ is unique.
  \item $s:\Hop \to H$ is a bialgebra homomorphism, \ie
    \suck
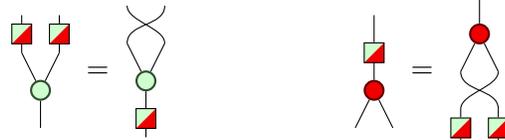
    \[
    \inltf{hopf-antihom-i}
    = \ \inltf{hopf-antihom-ii}
    \qquad\qquad\qquad
    \inltf{hopf-antihom-iii}
    = \ \inltf{hopf-antihom-iv}
    \]
    \sssuck
  \item $H^\sigma$ is a Hopf algebra if and only if $s$ is invertible, in
    which case the antipode of $H^\sigma$ is $s^{-1}$.
  \item If $H$ is commutative or cocommutative then $s\circ s = \id{H}$.
  \end{enumerate}
\end{proposition}

\begin{definition}\label{def:dual-hopf-alg}
  Let $(H,\gmult,\gunit,\rcomult,\rcounit,\antipode)$ be a Hopf
  algebra, and suppose that the object $H$ has a left dual $H^*$.  We
  define the \emph{dual Hopf algebra}
  $(H^*,\rcomult\!{}^*,\rcounit{}^*,\gmult\!{}^*,\gunit{}^*,\antipode{}^*)$ as :
  \[
  \rcomult\!{}^*:= \inltf{red-dual-mult}  \qquad
  \rcounit{}^*:= \inltf{red-dual-unit}  \qquad
  \gmult\!{}^*:= \inltf{green-dual-comult} \qquad
  \gunit{}^*:= \inltf{green-dual-counit} \qquad
  \antipode{}^* := \inltf{antipode-dual}
  \]
\end{definition}

\noindent
It's straightforward to prove that $H^*$ is indeed a Hopf algebra using
the equations of Def~\ref{def:duals}
In later sections it will be helpful to consider duals with
respect to different cups and caps, in which case we will vary
notation accordingly but the same construction is used in all cases.

% Special Frobenius algebras can be understood as arising from a
% distribution law of comonoids over monoids \cite{Lack:2004sf}.  In the
% other direction, distributing monoids over comonoids yields
% bialgebras.

% \begin{definition}\label{def:bialgebra}
%   A \emph{bialgebra} in symmetric monoidal category \catC consists of
%   a monoid and a comonoid on the same object, which jointly obey the
%   \emph{copy}, \emph{cocopy}, \emph{bialgebra}, and \emph{scalar} laws
%   depicted below.
%   \[
%   bialg laws
%   \]
%   Note that the dashed box above represents an empty diagram.
% \end{definition}

% \REM{def: Hopf Algebras}

% \REM{Properties of Hopf algebras}\TODO{What basic properties of Hopf
%   algebras will be needed?}

% \REM{remark : Scaled version}
% \REM{very short summary of \cite{Duncan2016Interacting-Fro}}

\section{Hopf-Frobenius Algebras}
\label{sec:when-hopf-algebras}

We now arrive at the main subject of this paper, Hopf-Frobenius
algebras in an arbitrary symmetric monoidal category \catC.  These
algebras generalise interacting Frobenius algebras
\cite{Coecke:2009aa,Duncan2016Interacting-Fro}, and share the same
gross structure.  It will be helpful to introduce a weaker notion
first.

\begin{definition}\label{def:pre-HF}
  A \emph{pre-Hopf-Frobenius algebra} or \emph{pre-HF algebra}
  consists of an object $H$ bearing a green monoid $(\gmult,\gunit)$,
  a green comonoid $(\gcomult,\gcounit)$, a red monoid
  $(\rmult,\runit)$, a red comonoid $(\rcomult,\rcounit)$ and an
  endomorphism $\antipode$ such that
  $(\gmult,\gunit,\gcomult,\gcounit)$ and
  $(\rmult,\runit,\rcomult,\rcounit)$ are Frobenius algebras, and
  $(\gmult,\gunit,\rcomult,\rcounit,\antipode)$ is a Hopf algebra.
\end{definition}

%% This amounts to not assuming that either of the Hopf algebras are
%% integral Hopf algebras.

\begin{definition}\label{def:HF-algebra}
  A \emph{Hopf-Frobenius algebra}, or \emph{HF algebra}, is a
  pre-Hopf-Frobenius algebra where $\antipode$ satisfies the left
  equation below,
  \[
  \hopfisfrob{antipode-form-2}\,
  \]
  and with $\antipoded$ defined as in the right equation above, 
  $(\rmult,\runit,\gcomult,\gcounit,\antipoded\;)$ is a Hopf algebra.
\end{definition}

We refer to the four algebras that make up an HF algebra by the
colour\footnote{If you are reading this document in monochrome
  \emph{green} will appear as light grey and \emph{red} as dark grey.}
of their \emph{multiplication}, so that
$(\gmult,\gunit,\rcomult,\rcounit,\antipode\;)$ is the \emph{green}
Hopf algebra, $(\rmult,\runit,\rcomult,\rcounit)$ is the \emph{red}
Frobenius algebra, \etc

% \begin{remark}\label{rem:antipodes-may-coincide}
%   Despite their distinct definitions, the two antipodes may coincide;
%   for example if both Frobenius algebras are symmetric, as in a group
%   algebra, then they are equal.
% \end{remark}

We now move on to the main topic of the section: under what conditions
does a Hopf algebra extend to a Hopf-Frobenius algebra?  Henceforward,
unless otherwise stated, \catC will denote a symmetric monoidal
category, and $H$ will denote a Hopf algebra
$(H,\gmult,\gunit,\rcomult,\rcounit,\antipode)$ in \catC.  Omitted
proofs are found in Appendix \ref{sec:proofs-omitted-from}.

A key concept is that of an integral. Pareigis \cite{PAREIGIS1971588}
proved\footnote{ This is a generalisation of earlier work by Larson
  and Sweedler \cite{Larson1969An-Associative-} showing that the space
  of integrals in \FVect is always isomorphic to $k$.}  that in
\FPMod, the category of finitely generated projective modules over a
commutative ring, a Hopf algebra has Frobenius structure when its
space of integrals is isomorphic to the ring.  More generally,
Takeuchi \cite{takeuchi1999finite} and Bespalov et
al. \cite{bespalov2000integrals} gave conditions for the space of
integrals in certain braided monoidal categories to be invertible.

\begin{definition}
  A \emph{left (co)integral} on $H$ is a copoint \gcoint$:H\rightarrow
  I$  (resp. a point \rint$:I \rightarrow H$), satisfying the equations:
\[
\begin{tikzpicture}
	\begin{pgfonlayer}{nodelayer}
		\node [style=green vertex] (0) at (-2.25, 0) {};
		\node [style=none] (2) at (-2, 0.5) {};
		\node [style=none] (3) at (-2.25, -0.5) {};
		\node [style=none] (7) at (-0.5, 0.75) {};
		\node [style=none] (8) at (-0.5, -0.75) {};
		\node [style=red vertex] (10) at (-0.5, 0.25) {};
		\node [style=none] (11) at (-1.5, 0) {=};
		\node [style=none] (14) at (1.25, -0.5) {};
		\node [style=none] (15) at (1, 0.5) {};
		\node [style=none] (18) at (2.75, -0.75) {};
		\node [style=none] (19) at (2.75, 0.75) {};
		\node [style=none] (22) at (1.75, 0) {=};
		\node [style=green vertex] (25) at (2.75, -0.25) {};
		\node [style=red vertex] (26) at (1, 0) {};
		\node [style=integral] (27) at (-2.5, 0.5) {};
		\node [style=integral] (28) at (-0.5, -0.25) {};
		\node [style=cointegral] (29) at (0.75, -0.5) {};
		\node [style=cointegral] (30) at (2.75, 0.25) {};
	\end{pgfonlayer}
	\begin{pgfonlayer}{edgelayer}
		\draw [style=directed edge, bend right=15] (27) to (0);
		\draw [style=directed edge, bend left=15] (2.center) to (0);
		\draw [style=directed edge] (0) to (3.center);
		\draw [style=directed edge] (7.center) to (10);
		\draw [style=directed edge] (28) to (8.center);
		\draw [style=directed edge, bend right=15] (26) to (29);
		\draw [style=directed edge, bend left=15] (26) to (14.center);
		\draw [style=directed edge] (15.center) to (26);
		\draw [style=directed edge] (25) to (18.center);
		\draw [style=directed edge] (19.center) to (30);
	\end{pgfonlayer}
\end{tikzpicture}
\]
A \emph{right (co)integral} is defined similarly.
\end{definition}

\begin{definition}
  An \emph{integral Hopf algebra} $(H,\rint,\gcoint)$ is a Hopf
  algebra $H$ equipped with a choice of right integral $\gcoint$, and
  left cointegral $\rint$, such that
  $\gcoint \circ \rint = \textrm{id}_I$.
\end{definition}

%% \begin{remark}
%%   An integral Hopf algebra coincides with the concept of when a Hopf
%%   algebra is \emph{CoFrobenius}, as established by Lin in
%%   [\REM{Semiperfect Coalgebras}]. We have refrained from using this
%%   language as it seems to imply that it is the dual notion of a
%%   Frobenius algebra. WE may see that the definition of a Frobenius
%%   algebra is self dual, as one can see from Definition
%%   \ref{def:frobenius-alg}.
%% \end{remark}

\begin{lemma}
\label{theorem:invertible-antipode}
  Let $(H, \rint, \gcoint)$ be an integral Hopf algebra. Then the
  following map is the inverse of the antipode.
\[
\begin{tikzpicture}
	\begin{pgfonlayer}{nodelayer}
		\node [style=red vertex] (0) at (-0.25, 0.25) {};
		\node [style=green vertex] (3) at (0.25, -0.25) {};
		\node [style=none] (4) at (-0.5, -0.25) {};
		\node [style=none] (5) at (0.5, 0.25) {};
		\node [style=integral] (6) at (-0.25, 0.5) {};
		\node [style=cointegral] (7) at (0.25, -0.5) {};
		\node [style=none] (9) at (-1, 0) {:=};
		\node [style=none] (10) at (-1.75, 0.5) {};
		\node [style=none] (11) at (-1.75, -0.5) {};
		\node [style=antipode] (12) at (-1.75, 0) {};
		\node [style=none] (13) at (-1.5, 0.25) {\scriptsize -1};
	\end{pgfonlayer}
	\begin{pgfonlayer}{edgelayer}
		\draw [style=directed edge] (6) to (0);
		\draw [style=directed edge] (3) to (7);
		\draw [style=directed edge, bend right=15] (0) to (4.center);
		\draw [style=directed edge, bend left=15] (5.center) to (3);
		\draw [style=directed edge, in=135, out=-30] (0) to (3);
		\draw (10.center) to (12);
		\draw (12) to (11.center);
	\end{pgfonlayer}
\end{tikzpicture}
\]
\end{lemma}

%%%% removed because immediate from previous
% \begin{corollary}
%     \[
% \tikzfig{hopf-is-frob/cup-cap-identity}\]
% \end{corollary}

%% \begin{corollary}\label{corollary:antipode-form}
%%     \[
%% \tikzfig{hopf-is-frob/antipodeform}\]
%% \end{corollary}

The statement of the above Lemma bares some similarities with
Definition \ref{def:duals}. In what follows, we will be generalising
this definition to capture the situation that arises with integral
Hopf algebras.
  \begin{definition}\label{def:half-dual}
  Let $A$ and $B$ be objects in a symmetric monoidal category \catC. A
  is a \emph{right half dual of $B$} if there exists morphisms
  $\hopfisfrob{triangle}:I\to A\otimes B$ and
  $\hopfisfrob{cotriangle}: B \otimes A \to I$ which satisfy the
  following equation
   \[
   \inlfropf{left-snake} = \ \inltf{idA}
   \]
  In this circumstance, $B$ is a \emph{left half dual of $A$}  
\end{definition}
  
Half duals are a strict generalisation of duals in the sense of
Definition~\ref{def:duals}.  Unlike true duals, an object
may have non-isomorphic half duals.  For example, if $B$ is left dual
to $A$, with a section $m:B\hookrightarrow C$ for some retraction
$m':C\twoheadrightarrow B$, then $C$ is a left half dual of
$A$. Further, any integral Hopf algebra $(H,\rint,\gcoint)$ makes $H$
left half dual to itself as follows.

\begin{definition}
 Let $(H, \rint, \gcoint)$ be an integral Hopf
  algebra, and define 
  \[
  \beta := \inltf{hopf-is-frob/cup1-def}
  \qquad\qquad
  \gamma := \inltf{hopf-is-frob/cap1-def}
  \qquad\qquad
  \beta' := \inltf{hopf-is-frob/cup2-def}
  \qquad\qquad
  \gamma' := \inltf{hopf-is-frob/cap2-def}
  \]
  With these definitions, $\gamma$ and $\beta$ make $H$ half dual to
  itself, and $\gamma'$ and $\beta'$ make $H$ half dual to itself but
  in a different way. We say that $H$ is \emph{nondegenerate} when
  $\gamma$ and $\beta$ are a cap and a cup respectively,
  (c.f. Definition \ref{def:comclcat}), making $H$ fully dual to itself. Furthermore,
  in this situation, $\gamma'$ and $\beta'$ are also a cup and a cap,
  giving a different self-dual structure to $H$.
\end{definition}

  \begin{lemma}\label{lem:equivalent-to-nondegenerate}
    Let $(H, \rint, \gcoint)$ be an integral Hopf algebra. $H$ is
    nondegenerate if and only if 
    \[
    \inltf{hopf-is-frob/equivalent-to-nondeg}\]
    
  \end{lemma}

\begin{lemma}\label{lem:nondegenerate-implies-preHF}
  Let $(H, \rint, \gcoint)$ be an integral Hopf
  algebra. $H$ is nondegenerate if and only if $\beta$ is a Frobenius
  form for $(H,\gmult,\gunit)$, or equivalently, if and only if
  $\gamma'$ is a Frobenius form for $(H,\rcomult,\rcounit)$. Hence, if
  $H$ is nondegenerate, then $H$ admits a pre-HF algebra structure.
\end{lemma}
\noindent
Per Definition~\ref{def:frob-alt}, $\gcoint$ is the counit of the green
Frobenius algebra and the green comultiplication is obtained by dualising
$\gmult$ with $\beta$. The red unit and multiplication are obtained in
a similar manner.

\begin{definition}\label{def:integral-section}
  Let the object $H$ have a right half dual $H^*$. The
  \emph{integral morphism} $\Isec: H \rightarrow H$ is defined as shown
  below.
  \[
  \hopfisfrob{integral-projection}
  \]
  Note that this definition does \emph{not} depend on the choice of
  half dual -- see Lemma~\ref{lem:I-ind-of-half-dual}
\end{definition}

If $H$ is in \FPMod, then \Isec may be seen as a map from $H$ to the
space of left cointegrals. In fact, it is the retraction of the natural
injection from the space of left integrals into $H$. As such, it acts
trivially on integrals, and for every element $v \in H$, $\Isec(v)$ is
a left integral (which may be 0). In
Lemma~\ref{lem:I-preserves-integrals} we show that this holds in the
general case, where we have exchanged elements of a module for points
$p:I \rightarrow H$.

\begin{lemma}\label{lem:I-preserves-integrals}
  Given a point $p:I \rightarrow H$, and copoint $q:H \rightarrow I$,
  the morphism $\Isec \circ p$ is a left cointegral, and
  $q \circ \Isec$ is a right integral. In addition, $p$ is a left
  cointegral if and only if $\Isec \circ p = p$, and $q$ is a right
  integral if and only if $q \circ \Isec = q$.
%% \[ %% \hopfisfrob{projection-acts-as-integral}
%% \]
\end{lemma}

%% Before moving on, lets look at what each of the parts of this map are
%% doing in \FVect

%%  \hopfisfrob{star-action}  is taking an element $x \in H$ and mapping that to
%% the map $H \rightarrow H$ defined as $a \mapsto xs(a)$, where $s$ is
%% the antipode. While this is not an element of $H^* \otimes H$, by the
%% compact structure of $\mathbf{FVect}$, it does correspond bijectively
%% to an element of $H^* \otimes H$

%% \hopfisfrob{star-coaction} is taking a functional $f \in H^*$ and an element
%% $x \in H$ and mapping that to
%%     \[\sum f(s(x_{(1)}))x_{(2)}
%%     \]
%%     where $s$ is the antipode, and we have used Sweedler notation

\begin{definition}\label{def:frob-cond}
  We say that a Hopf algebra satisfies the \emph{Frobenius condition}
  if there exists maps \rint and \gcoint such that
  \[
  \begin{tikzpicture}
	\begin{pgfonlayer}{nodelayer}
		\node [style=none] (6) at (-1.75, 0) {=};
		\node [style=none] (9) at (-1, -1) {};
		\node [style=none] (10) at (-1, 1) {};
		\node [style=none] (11) at (0, 0) {and};
		\node [style=none] (14) at (1.5, 0) {=};
		\node [style=cointegral] (19) at (-1, 0.5) {};
		\node [style=integral] (20) at (-1, -0.5) {};
		\node [style=integral] (21) at (1, 0.25) {};
		\node [style=cointegral] (22) at (1, -0.25) {};
		\node [style=none] (37) at (2, 0.25) {};
		\node [style=none] (38) at (2.5, 0.25) {};
		\node [style=none] (39) at (2, -0.25) {};
		\node [style=none] (40) at (2.5, -0.25) {};
		\node [style=none] (54) at (-3, -1.25) {};
		\node [style=none] (55) at (-3.25, -1.25) {};
		\node [style=none] (56) at (-2.5, 1.25) {};
		\node [style=none] (57) at (-2.25, 1.25) {};
		\node [style=green vertex] (58) at (-2.75, 0.5) {};
		\node [style=red vertex] (59) at (-2.75, -0.5) {};
		\node [style=antipode] (60) at (-2.5, 1) {};
		\node [style=antipode] (61) at (-3, -1) {};
		\node [style=none] (62) at (-3.25, 1) {};
		\node [style=none] (63) at (-2.25, -1) {};
		\node [style=none] (64) at (-2.75, 0) {};
		\node [style=none] (65) at (-2.75, -0.25) {};
	\end{pgfonlayer}
	\begin{pgfonlayer}{edgelayer}
		\draw [style=directed edge] (10.center) to (19);
		\draw [style=directed edge] (20) to (9.center);
		\draw [style=directed edge] (21) to (22);
		\draw [style=dashed] (37.center) to (39.center);
		\draw [style=dashed] (39.center) to (40.center);
		\draw [style=dashed] (40.center) to (38.center);
		\draw [style=dashed] (38.center) to (37.center);
		\draw [in=90, out=90, looseness=1.75] (56.center) to (57.center);
		\draw (59) to (61);
		\draw (58) to (60);
		\draw [style=directed edge, bend left=15] (59) to (63.center);
		\draw [style=directed edge, bend right=15] (62.center) to (58);
		\draw [style=directed edge] (58) to (65.center);
		\draw [style=directed edge] (65.center) to (59);
		\draw [style=directed edge, in=90, out=-90] (56.center) to (60);
		\draw [style=directed edge, in=-90, out=45] (64.center) to (57.center);
		\draw (57.center) to (56.center);
		\draw [in=-90, out=-90, looseness=1.75] (54.center) to (55.center);
		\draw (55.center) to (54.center);
		\draw [style=directed edge] (61) to (54.center);
		\draw [style=directed edge, in=-135, out=90] (55.center) to (64.center);
	\end{pgfonlayer}
\end{tikzpicture}
 \]
\end{definition}

% \begin{corollary}\label{corr:frob-cond-implies-integral}
%   If $H$ satisfies the Frobenius condition, then $(H, \rint, \gcoint)$
%   is an integral Hopf algebra.  
% \end{corollary}
% \begin{proof}
%   The Frobenius condition implies that $\Isec\circ\rint=\rint$, and
%   $\gcoint\circ\Isec=\gcoint$. Hence, by
%   Lemma~\ref{lem:I-preserves-integrals}, \rint is a left cointegral
%   and \gcoint is a right integral.
% \end{proof}

% \ROUGH{
%   We will shortly be proving that a Hopf algebra will fulfill the
%   Frobenius condition if and only if it also admits a pre-HF
%   algebra. By Lemma \ref{lem:I-preserves-integrals}, recall that if we
%   have an integral $p:I \rightarrow H$, then $\Isec\circ p=p$. The
%   Frobenius condition implies, therefore, that $p = \rint \circ \gcoint
%   \circ p$. Hence, $p$ will be a scalar multiple of $\rint$. This is
%   captured in the following Lemma.

%   We will shortly be showing that when a Hopf algebra fulfills the
%   Frobenius condition, it admits a Frobenius algebra. It is a classic
%   result by Pareigis \cite{PAREIGIS1971588} that in \FPMod, a Hopf
%   Algebra will admit a Frobenius algebra when integrals only differ by
%   a scalar multiple.

% }

  % \begin{lemma}\label{lem:equaliser}
  %   Let the object $H$ have a right half dual $H^*$, where $H$ is a
  %   Hopf algebra. $H$ fulfills the Frobenius condition if and only if
  %   there is an equaliser \rint of
  %   \[
  %   \inlfropf{equaliser-statement}
  %   \]
  %   if and only if there is a coequaliser \gcoint of
  %   \[
  %   \inlfropf{equaliser-statement-2}
  %   \]
    
  % \end{lemma}

\begin{theorem}\label{theorem:frob}
  If $H$ satisfies the Frobenius condition, then $H$ admits a pre-HF
  algebra structure with the Frobenius forms and their inverses as shown below.
  \[
  \inltf{green-cup} := \inltf{hopf-is-frob/cup1-def} \qquad
  \inltf{green-cap} := \inltf{hopf-is-frob/cap1-def} \qquad
  \qquad
  \inltf{red-cap} := \inltf{hopf-is-frob/dual-cap-def} \qquad
  \inltf{red-cup} := \inltf{hopf-is-frob/dual-cup-def} \qquad
  \]
  Further, $(H,\rint,\gcoint)$ is an integral Hopf algebra.
\end{theorem}
\begin{proof}
  The Frobenius condition implies that $\Isec\circ\rint=\rint$, and
  $\gcoint\circ\Isec=\gcoint$. Hence, by
  Lemma~\ref{lem:I-preserves-integrals}, \rint is a left cointegral
  and \gcoint is a right integral. Now, we only need to show that it
  is nondegenerate. Observe that

  \[
  \hopfisfrob{frob-condition-proof-1}
  \]

  where $(1.)$ is due to the Frobenius condition, $(2.)$ comes from
  associativity and $(3.)$ comes from the fact that the antipode is a
  bialgebra homomorphism $\Hop \rightarrow H$. The presence of half
  duals gives us $(4.)$, and $(5.)$ is due to the Hopf law. We then
  get the following identity
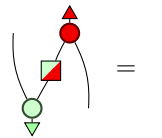
\[
\hopfisfrob{Frob-cup-cap}
\]
This is the identity required to make $(H, \rint, \gcoint)$
nondegenerate, and we have our result by Lemma
\ref{lem:nondegenerate-implies-preHF}.
\end{proof}
The explicit definitions of the green comonoid and red monoid
structures are shown below.
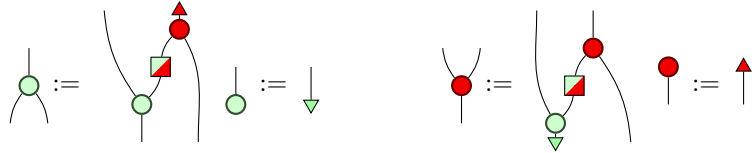
\[
\hopfisfrob{def-of-new-operations-2}
\]
\suck

As the name suggests, $H$ fulfilling the Frobenius condition is
equivalent to $H$ admitting a Frobenius algebra structure. To prove
this, we must first prove the following intermediate lemma
  
\begin{lemma}\label{lem:equaliser}
    Let the object $H$ have a right half dual $H^*$, where $H$ is a
    Hopf algebra. $H$ fulfills the Frobenius condition if and only if
    there is an equaliser \rint of
    \[
    \inlfropf{equaliser-statement}
    \]
    if and only if there is a coequaliser \gcoint of
    \[
    \inlfropf{equaliser-statement-2}
    \]
    
  \end{lemma}

  % It should be noted that in the proof of the above Lemma, we show that if $H$ fulfills
  % the Frobenius condition, then $\rint$ is a split equaliser.

  \begin{theorem}\label{theorem:admits-Frobenius}
  Let $H$ be a Hopf algebra. $H$ satisfies the Frobenius condition if
  and only if $H$ admits a Frobenius structure on its multiplication
  or its comultiplication. Hence, $H$ fulfills the Frobenius condition
  if and only if it admits a pre-HF algebra structure.
  \end{theorem}
  \begin{proof}
    Clearly if $H$ satisfies the Frobenius condition, then by Theorem
    \ref{theorem:frob} it admits a Frobenius structure. For the
    converse, suppose that $H$ admits a Frobenius structure
    $(H,\gmult, \gunit, \gcomult, \gcounit)$ on its
    multiplication. This provides a cup and a cap that makes $H$ self
    dual. Set $\alpha:=\inlfropf{pre-integral}$. We will show that
    $\alpha:I \rightarrow H$ is a split equaliser of the diagram
    \[
      \inlfropf{admits-Frobenius-proof-1c}.
    \]
    
  Note that the lower morphism is simply \gcomult. To show that
  $\alpha$ is a split equaliser, we must first show
  that it is is a cone of the appropriate diagram. This follows from
  the properties of the Frobenius algebra
  \[
  \hopfisfrob{admits-Frobenius-proof-2}
  \]
  We now need to find a retract of $\alpha$ and
  \gcomult. It is clear that
  \gcounit is a retract of $\alpha$, and
  \inlfropf{admits-Frobenius-proof-3} is a retract of
  \gcomult. The final condition for $\alpha$ to be a
  split equaliser is
  \[
  \hopfisfrob{admits-Frobenius-proof-4}
  \]
  
  Thus, $\alpha$ a split equaliser. By Lemma \ref{lem:equaliser}, this
  implies that $H$ must satisfy the Frobenius condition. If $H$ admits
  a Frobenius algebra on its comultiplication, then the same result
  holds by duality. Hence, when $H$ admits a pre-HF algebra structure,
  it fulfills the Frobenius condition, and by Theorem
  \ref{theorem:frob}, we get our equivalence.

\end{proof}

  It may be important to mention that if $H$ admits a Frobenius
  structure, then this is not necessarily the same structure as the
  one given by Theorem~\ref{theorem:frob}. In Proposition
  \ref{prop:frob-structs-are-invertible-elts}, we show that Frobenius
  structures are not unique. One may start with different Frobenius
  structures, and end up with the same pre-HF algebra. In the proof of
  Lemma \ref{theorem:admits-Frobenius}, one constructs an integral and
  cointegral, and it is these that determine the appropriate Frobenius
  structure.\\

%% \begin{corollary}\label{cor:frobcond-implies-preHF}
%% Every Hopf algebra which satisfies the Frobenius condition defines a
%% pre-HF algebra.    
%% \end{corollary}

Pareigis \cite{PAREIGIS1971588} showed that in \FPMod, a Hopf
  Algebra will admit a Frobenius algebra when integrals only differ by
  a scalar multiple. This is clear from Lemma \ref{lem:equaliser}.
  Under mild assumptions, this is equivalent to the Frobenius
  condition.

In a monoidal category, an object $A$ is said to have \emph{enough
  points} if, for all morphisms $f,g:A\to B$, we have
\[
(\forall x : I\to A, \ \ fx = gx) \Rightarrow f = g\;.
\]

\begin{lemma}
\label{lem:integrals-are-multiples}
  Let $(H,\rint,\gcoint)$ be an integral Hopf algebra and suppose that
  $H$ has enough points. If every left cointegral (right integral) is
  a scalar multiple of \runit (resp. \gcounit) then $H$ fulfills the
  Frobenius condition
\end{lemma}

%%%% Merged with previous.
% \begin{lemma}
%   Let \catC have enough points, and $(H,\rint,\gcoint)$ be an integral
%   Hopf algebra. Suppose that every left integral (right cointegral) is
%   a scalar multiple of \rint (resp. \gcoint). Then $H$ fulfills the
%   Frobenius condition.
% \end{lemma}

Since \FPMod (and \FVect) are categories where every object has enough
points, Lemma~\ref{lem:integrals-are-multiples} implies Pareigis'
condition is exactly the Frobenius condition.

We may now consider the main theorem of the paper - when exactly does
a Hopf algebra admit a Hopf-Frobenius algebra?

\begin{theorem}\label{thm:iff-frob-condition}
  Let $H$ be a Hopf algebra such that the object $H$ has some weak
  right dual $H^*$. Then $H$ admits a Hopf-Frobenius algebra structure
  if and only if $H$ fulfills the Frobenius condition.
\end{theorem}
\begin{proof}[Sketch of Proof.]  
  We explore this in full detail in the appendix. Here, we only
  outline a sketch of the proof

  If $H$ is a Hopf-Frobenius algebra, then it admits a Frobenius
  algebra, and therefore, by Theorem \ref{theorem:admits-Frobenius}, it
  fulfills the Frobenius condition.
  
  Consider the converse. By Theorem \ref{theorem:frob}, we know that
  if $H$ fulfills the Frobenius condition, then $H$ admits a pre-HF
  algebra and $(H,\runit,\gcounit)$ is an integral Hopf algebra. It
  follows from Lemma~\ref{theorem:invertible-antipode} that
  $\antipode = \antipodeform$, and we show in Lemma
  \ref{lemma:iff-integral} that this is true if and only if
  $(H,\runit,\gcounit)$ is an integral Hopf algebra. Hence, $H$ admits
  Hopf-Frobenius structure if and only if
  $(H,\rmult,\runit,\gcomult,\gcounit,\antipoded)$ forms a Hopf
  algebra, where $\antipoded =\daggerantipode$. We begin by proving
  that $(H,\rmult,\runit,\gcomult,\gcounit)$ is a bialgebra, and then
  that $\antipoded =\daggerantipode$ is the appropriate antipode to
  make this bialgebra a Hopf algebra.

  $H$ admits pre-HF algebra structure, so it has a structure that
  makes $H$ self dual. Let $(\cdot)\gtrans$ be the duality defined by
  the green Frobenius algebra. The dual of a Hopf algebra is a Hopf
  algebra, in the sense of Definition
  \ref{def:dual-hopf-alg}. Therefore, applying the dual to $H$ will
  give us another Hopf algebra. Lemma
  \ref{lemma:how-green-relates-to-dagger} tells us that
  \suck
  \[
  \left( \inltf{green-mult}\!\!\right)\gtrans =
  \inltf{green-comult}
  \qquad\qquad
  \left( \inltf{red-comult}\!\!\right)\gtrans =
  \inltf{red-mult-swap}
  \]
  \suck
  
  Set
  $H^{\gtrans}:=(H,\rcomult\!{}\gtrans,\rcounit{}\gtrans,\gmult\!{}\gtrans,\gunit{}\gtrans,\antipode{}\gtrans)$
  to be the Hopf algebra obtained when we apply $(\cdot)\gtrans$ to
  $H$. The above result tells us that
  $(H,\rmult,\runit,\gcomult,\gcounit)$ is equal to
  $(H^{\gtrans})^\sigma$ when viewed as a bialgebra. Therefore, by
  Proposition \ref{prop:antipode-props} we only need to show that
  $\antipode\gtrans$ has an inverse. But the duality operation,
  $(\cdot)\gtrans$, maps isomorphisms to isomorphisms, so since
  $\antipode$ is invertible, $\antipoded := (\antipode^{-1})\gtrans$
  will be the antipode of $(H,\rmult,\runit,\gcomult,\gcounit)$. All
  that is left is to show that $\antipoded =\daggerantipode$, and this
  is accomplished by simple calculation.

  % For the other direction, we only need to show that
  % \[
  %   \hopfisfrob{HF-frob-condition}
  % \]
  % This follows from the fact that $\antipode = \antipodeform$ and
  % $\antipoded =\daggerantipode$, and the Hopf law.
\end{proof}

Let us summarise the various equivalent conditions for a Hopf algebra
to be Hopf-Frobenius.

\begin{theorem}
  Let $H$ be a Hopf algebra. The following conditions are equivalent
  \begin{itemize}
  \item $H$ admits a Hopf-Frobenius algebra structure
  \item $H$ admits a pre-HF algebra structure
  \item $H$ fulfills the Frobenius condition
  \item $H$ admits a Frobenius algebra structure on the multiplication
    or the comultiplication
  \item $H$ admits an equaliser \rint of 
    \[
    \inlfropf{equaliser-statement}
    \]
  \item $H$ admits an integral algebra structure, $(H, \rint,
    \gcoint)$, and $H$ is nondegenerate
  \item $H$ admits an integral algebra structure, $(H, \rint,
    \gcoint)$, and $\gcoint \circ \antipode \circ \rint = 1_I$
  \end{itemize}
\end{theorem}

We finish this section by asking how canonical this structure
is. Frobenius structure in general is non-canonical (cf. Proposition
\ref{prop:frob-structs-are-invertible-elts}). Despite this, we find that
Hopf-Frobenius structure is canonical, as follows.
\begin{lemma}\label{lem:HF-is-unique}
    Let $H$ admit a Hopf-Frobenius algebra structure. Then this
    structure is unique up to invertible scalar.
\end{lemma}

\section{Examples}
\label{sec:examples}

Combined with the results of Larson and Sweedler
\cite{Larson1969An-Associative-}, Pareigis \cite{PAREIGIS1971588}, and
Lemma~\ref{lem:integrals-are-multiples},
Theorem~\ref{thm:iff-frob-condition} implies that any Hopf algebra in
\FVect is Hopf-Frobenius.  This allows the direct extension of
\cite{Duncan2016Interacting-Fro} to non-abelian group algebras, but
there are plenty of other examples.  We briefly mention some examples
which are neither commutative nor cocommutative.

\begin{example}
  Let $k$ be a field with a primitive $n^{th}$ root of unity $z$.  The
  \emph{Taft Hopf algebras} \cite{Taft2631} are a family of Hopf
  algebras in \FVect whose antipodes have order $2n$.  Generically,
  the algebra $(H,\mu,1,\Delta,\epsilon,s)$ is generated by elements $x$
  and $g$, such that $x^n = 0$, $g^n=1$, and $gx = zxg$. The coalgebra
  is defined $\Delta(x) = 1 \OX x + x \OX g$, and $\Delta(g)=g \OX g$,
  with $\epsilon(x)=0$ and $\epsilon(g)=1$. The antipode is
  $s(x)=-xg^{-1}$, $s(g)=g$, and the rest of the structure follows
  from the Hopf algebra axioms. We may see that $H$ has the basis
  $x^\alpha g^\beta$, where $0 \leq \alpha, \beta, \leq n-1$, so this
  will imply that $H$ is $n^2$ dimensional. One can calculate that the
  left integral of $H$ is
  \[
  \sum^n_{i=1} z^{-i} g^i x^{n-1}
  \]
  and the right cointegral is the functional that takes $x^{n-1}$ to 1
  and every other basis element to 0. We explicitly construct the HF
  algebra of the Taft Hopf algebra when $n=2$ in the appendix.
\end{example}

\begin{example}
  Hopf algebras which arise as the quantum enveloping algebra of Lie
  algebras are a type of quantum group. Since these are infinite
  dimensional, they cannot be Hopf-Frobenius algebras.  However their
  finite dimensional quotients will be Hopf-Frobenius.  See Kassel
  \cite{kassel2012quantum} for an example.
\end{example}

\noindent
Moving away from \FVect, we consider \catRel, the category of sets and
relations.

\begin{example}
  Let $G$ be an infinite group.  Following Hasegawa
  \cite{hasegawa2010bialgebras} we can construct its group algebra in
  \catRel.  The integral is $\{ (\star, g) \;|\; g \in G\}$ and the
  cointegral is the singleton $(1,\star)$.  The construction detailed
  in Theorem \ref{theorem:frob} recovers the expected 
  multiplication and comultiplication relations:
  \suck 
  \begin{gather*}
  \gcomult:= a \mapsto (b,c) \text{ such that } a=bc \\
  \rmult:= (a,b) \mapsto \begin{cases}
    a           &\text{ if } a=b\\
    \emptyset   &\text{ otherwise}
  \end{cases}
  \end{gather*}
\end{example}
\suck

\noindent
We look forward to discovering more exotic examples.

 %% \begin{center}
%%       %% \begin{table}[H]
%% %% \begin{tabular}{l|llll}
%% %% $\gmult$ & $1$  & $x$  & $g$   & $gx$ \\ \hline
%% %% $1$      & $1$  & $x$  & $g$   & $gx$ \\
%% %% $x$      & $x$  & $0$  & $-gx$ & $0$  \\
%% %% $g$      & $g$  & $gx$ & $1$   & $x$  \\
%% %% $gx$     & $gx$ & $0$  & $-x$  & $0$ 
%% %% \end{tabular}
%% %%   \end{table}  \begin{table}[H]
%% %% \begin{tabular}{l|l}
%% %% $\rcomult$ &                      \\ \hline
%% %% $1$        & $1 \otimes 1$        \\
%% %% $x$        & $1 \OX x + x \OX g$  \\
%% %% $g$        & $g \OX g$            \\
%% %% $gx$       & $g \OX gx+ gx \OX 1$
%% %% \end{tabular}
%%     %% \end{table}

%%     \begin{table}[h]
%% \begin{tabular}{l|llllll|lll|lll|l}
%% $\gmult$ & $1$  & $x$  & $g$   & $gx$ &  & $\rcomult$ &                      &  & $\rcounit$ &     &  & $\antipode$ &      \\ \cline{1-5} \cline{7-8} \cline{10-11} \cline{13-14} 
%% $1$      & $1$  & $x$  & $g$   & $gx$ &  & $1$        & $1 \otimes 1$        &  & $1$        & $1$ &  & $1$         & $1$  \\
%% $x$      & $x$  & $0$  & $-gx$ & $0$  &  & $x$        & $1 \OX x + x \OX g$  &  & $x$        & $0$ &  & $x$         & $gx$ \\
%% $g$      & $g$  & $gx$ & $1$   & $x$  &  & $g$        & $g \OX g$            &  & $g$        & $1$ &  & $g$         & $g$  \\
%% $gx$     & $gx$ & $0$  & $-x$  & $0$  &  & $gx$       & $g \OX gx+ gx \OX 1$ &  & $gx$       & $0$ &  & $gx$        & $-x$
%% \end{tabular}
%% \end{table}
%% \end{center}

\section{A simpler Drinfeld double}
\label{sec:drinfeld-double}

\REM{this first para is some pretty weak sauce.  Needs improvement!}
Braided categories of modules over a Hopf algebra are widely used in
physics, where they give solutions to the Yang-Baxter equation and in
low dimensional topology, where they are used to find invariants.
However the category of modules over a Hopf algebra is braided if and
only if the Hopf algebra is \emph{quasi-triangular}.  The
\emph{Drinfeld double} \cite{drinfeld1986quantum} is a construction
that takes a Hopf algebra $H$ in \FVect, and produces a
quasi-triangular Hopf algebra $D(H)$ on the object $H \OX H^*$.  In
this section we use the self-duality of a Hopf-Frobenius algebra to
construct the canonical isomorphism $H \cong H^*$ and thus define a
simpler version of the Drinfeld double on $H\OX H$.

We will assume that \catC is a compact closed category. We denote the
green and red Hopf algebras of $H$ as $H\gconj$ and $H\rconj$
respectively.  We use the generalisation of Drinfeld's original
construction to symmetric monoidal categories, due to Chen
\cite{chen2000quantum}.

%% Consider a HF algebra $H$
%% \[
%% \hopfisfrob{the-lattice}
%% \]

%% The defintion of a HF algebra is completely symmetric. This makes it
%% difficult to specify specific parts of the Hf algebra. As such, we
%% will use the following convention.
%% \begin{definition}
%%   The \emph{green Hopf algebra}, denoted as $(H, \gmult)$ is the Hopf
%%   algebra within $H$ that has the green multiplication. Similarly, the
%%   \emph{red Hopf algebra}, denoted as $(H, \rmult)$ is the Hopf
%%   algebra that has the red multiplication.
%% \end{definition}

%% Let \HF(\catC) be the category of Hopf Frobenius algebras on an
%% autonomous category \catC, where the objects are HF algebras and the
%% morphisms are Hopf algebra morphisms on the main Hopf algebra. Note
%% that there is a bijective correspondance between morphisms between the
%% main Hopf algebra and its dagger.

\begin{definition}\label{def:nat-iso}
Let $H$ be a HF algebra on \catC. By Proposition
\ref{prop:duals-are-unique}, we may define an isomorphism
$\rediso:H\rightarrow H^*$, with inverse $\redinv:H^* \rightarrow H$
as

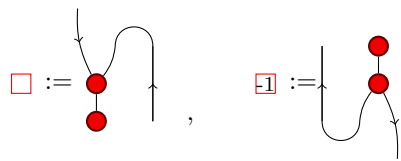
\[
\hopfisfrob{red-nat-iso}
\]
\end{definition}

\begin{lemma}\label{lemma:rediso}
  The morphism \rediso is a Hopf algebra homomorphism between
  $H_{\rconj}^\sigma$ and $H_{\gconj}^*$.
\end{lemma}

\begin{remark}
  % The morphism \rediso is a canonical natural isomorphism. The essential idea is as
  % follows: Given a compact closed category \catC, we consider the
  % category of HF algebras on \catC. We have two dual structures - the
  % red dual from the Hopf-Frobeius structure, and the dual from the
  % compact closed structure of \catC. By Lemma~\ref{lem:HF-is-unique},
  % since the Hopf-Frobenius structure is canonical, the red and green
  % Frobenius structures are also canonical, and by extension, the red
  % and green dual structures on $H$ are also canonical. Thus, given a
  % Hopf-Frobenius algebra on a compact closed c We find that
  % \rediso is a natural isomorphism between the functors induced from
  % these dual structures.

  The morphism $\rediso$ is the canonical isomorphism between the
  compact closed structure and the red dual structure given to us by
  the Hopf-Frobenius structure, in the sense of
  Proposition~\ref{prop:duals-are-unique}. By
  Lemma~\ref{lem:HF-is-unique}, since the Hopf-Frobenius structure is
  canonical, the red and green Frobenius structures are also
  canonical, and by extension, the red and green dual structures on
  $H$ are also canonical. Therefore, whenever $H$ admits a
  Hopf-Frobenius structure on a compact closed category, we may
  construct $\rediso$ up to a unique invertible scalar.
  
\end{remark}

\begin{definition}
  A Hopf algebra $H$ is \emph{quasi-triangular} if there exists a
  \emph{universal $R$-matrix} $R:I \rightarrow H \otimes H$ such that
\begin{itemize}
\item $R$ is invertible with respect to \gmult
\item \inlfropf{quasi-commutative}
\item \inlfropf{quasi-triangular}
\end{itemize}
\end{definition}
All cocommutative Hopf algebras are quasi-triangular, with $\rcounit
\otimes \rcounit$ as the universal $R$-matrix.  This definition is
motivated by the following theorem \cite{kassel2012quantum}.

\begin{theorem}
  The category of modules over a Hopf algebra is braided if and only
  if the Hopf algebra is quasi-triangular.
\end{theorem}

\begin{definition}
Let $H$ be a Hopf algebra in \catC with an invertible antipode. The
\emph{Drinfeld double} of $H$, denoted $D(H) = (H \otimes H^*,
\mu,1,\Delta,\epsilon,s)$, is a Hopf algebra defined in the following
manner:
\[
\hopfisfrob{D-double-structure}
\]
\end{definition}
%This uses citations

%% Doublecrossed biproducts in braided tensor categories, S. Zhang and H.X.Chen
%% This is a Preprint, refered to in the paper below

%% Hui-Xiang Chen (2000)Quantum doubles in monoidal categories,
%% Communications in Algebra, 28:5, 2303-2328, DOI: 10.1080/00927870008826961

\begin{theorem}[Drinfeld\cite{drinfeld1986quantum,chen2000quantum}]
  $D(H)$ is quasi-triangular, with the universal $R$-matrix shown below.
  \[
  \hopfisfrob{D-double-R-matrix}
  \]
\end{theorem}

Our goal is to use the Hopf-Frobenius structure to get a Hopf algebra
that is isomorphic to the Drinfeld double, but is easier to do
diagrammatic reasoning with.

We will now use the Hopf-Frobenius structure to derive a Hopf algebra
isomorphic to the Drinfeld double.  Consider the composite of the map
$1\otimes\rediso$ with the multiplication of the Drinfeld double:

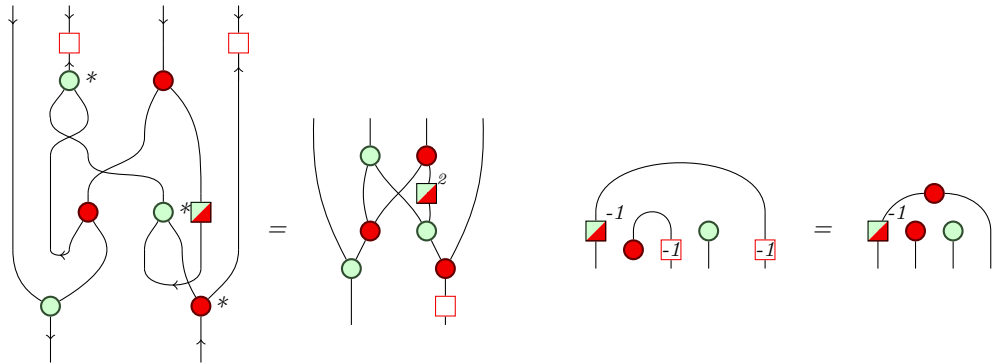
\begin{lemma}\label{lemma:rho-with-D-double}
  \[
  \hopfisfrob{rho-with-D-double}
  \]
\end{lemma}

%(H,\gmult,\gunit,\rcomult,\rcounit,\antipode)
\begin{definition}
  Let $H$ be a HF algebra. The \emph{red Drinfeld double}, denoted
  $D\rconj(H) = (H \otimes H,
  \mu\rconj,1\rconj,\Delta\rconj,\epsilon\rconj,s\rconj)$, is a Hopf
  algebra on the object $H \otimes H$ with structure maps
  \[
  \hopfisfrob{nice-structure-maps}\]
\end{definition}

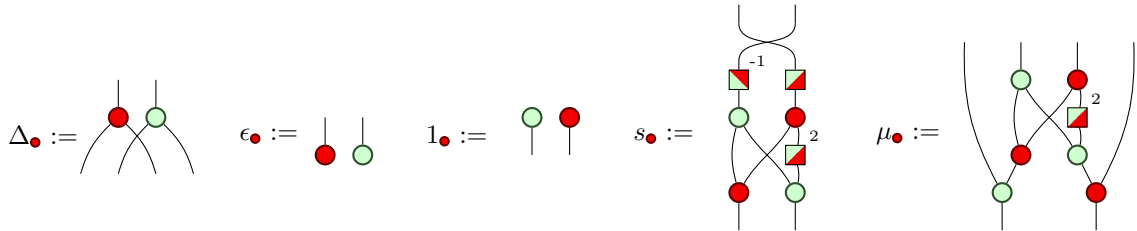
\begin{corollary}
  $D\rconj(H)$ is a quasi-triangular Hopf algebra isomorphic to the
  Drinfeld double, with universal R-matrix
  \[
  \hopfisfrob{dag-R-matrix}
  \]
  \suck
\end{corollary}

\section{Conclusion and further work}
\label{sec:concl-furth-work}

%%%% Points for new conclusion 
% * Move last paragraph of intro to end
% * Question : what are the necessary and sufficient conditions for
% integrals to exits?  For the Frob condition to be satisfied?
% * What about coFrobenius?
% * generalisations to planar or braided setting
% * applications to topological quantum computing
% * Encoding the group algebra of S_3 in \zxcalculus? 

We have generalised the notions of interacting Frobenius algebras
\cite{Coecke:2009aa,Duncan2016Interacting-Fro} and interacting Hopf
algebras \cite{Bonchi2014aJournal} to the non-commutative case, and in
the process shown that they are rather common structures.
This work could be viewed as an extension of classical results showing
that concrete Hopf algebras over finite dimensional vector spaces are
also Frobenius algebras \cite{Larson1969An-Associative-}.  Another
perspective is that we make precise how much ambient symmetry is
required to obtain a Hopf-Frobenius algebra.  The original setting of
interacting Frobenius algebras \cite{Coecke:2009aa} was a
$\dag$-compact category, which provides a lot of duality on top of the
commutative algebras themselves.  We show that none of this structure
is necessary: all that is required is one-sided \emph{half-dual} for
the carrier object.  
The major question that remains is to pin down exactly when the
Frobenius condition holds; as Lemma~\ref{lem:integrals-are-multiples}
shows, this is tightly related to the existence of integrals.  Compact
closure does not suffice to guarantee this: in \FPMod there are Hopf
algebras which are not Frobenius.  \REM{Unimodularity something
  something.}  

While we have established that Hopf algebras are frequently
Hopf-Frobenius, the resulting Frobenius algebras need not be well
behaved (commutative, dagger, special) as in the original quantum
setting \cite{PavlovicD:MSCS08}.  It remains to investigate what
Frobenius structures arise from ``interesting'' Hopf algebras, and
whether they have any application in the categorical quantum mechanics
programme, or conversely, how HF algebras may be applied in the study
of quantum groups.  Weaker structures such as the ill-named
co-Frobenius algebras or the stateful resource calculus of Bonchi et al
\cite{Bonchi:2019:DAL:3302515.3290338} perhaps offer an alternative to
the nonstandard approach \cite{Gogioso2016Infinite-dimens} to study infinite
dimensional systems. Beyond this, natural generalisations to the
braided or planar cases suggest themselves, although this will push
diagrammatic reasoning to its limits.

Our new Drinfeld double construction suggests that HF algebras could
find applications in topological quantum computation, particularly for
error correcting codes, an area where the \zxcalculus is already used
\cite{Beaudrap2017The-ZX-calculus}.  The smallest non-abelian group is
$S_3$, whose group algebra fits in 3 qubits with room to spare.

%%  Using $1 \otimes \rho_{\rtrans}$, we are able to define 

%% Now, consider the Hopf algebra with the same comultiplication, unit
%% and counit as $H \otimes H^{\dagger}$, and the multiplication map

%% \[
%% \hopfisfrob{nice-drinfeld-double}
%% \]

%% Let us denote this by $D_{\dag}(H)$

%%   \begin{lemma}
%%     $D_{\dag}(H)$ is naturally isomorphic to $D(H)$, with the isomorphism $1 \otimes \rho_{\rdot}$
%%   \end{lemma}

%% Using the natural isomorphism, we are able to show that $D_{\dag}(H)$
%% is also quasi-triangular, with universal $R$-matrix
%% \[
%% \hopfisfrob{dag-R-matrix}
%% \]

%This section will have the stuff on quasi-special stuff, the
%symmetric equivalence, how symmetric gives us a quasi triangular thing

%How about this for a structure
%First off, special. It doesn't really fit in anywhere, so It can
%probably be just at the start
%what unimodular and co-unimodular are equivalent to
%What symmetric is equivalent to
%Show that semisimple and cosemisimple in a Field of characteristic
%zero imply symmetric
%Symmetric gives a canonical quasi-triangular Hopf algebra.

%Finally, we have the drinfeld double.

\small
\bibliography{hopffrob}

\begin{thebibliography}{10}
\providecommand{\bibitemdeclare}[2]{}
\providecommand{\surnamestart}{}
\providecommand{\surnameend}{}
\providecommand{\urlprefix}{Available at }
\providecommand{\url}[1]{\texttt{#1}}
\providecommand{\href}[2]{\texttt{#2}}
\providecommand{\urlalt}[2]{\href{#1}{#2}}
\providecommand{\doi}[1]{doi:\urlalt{http://dx.doi.org/#1}{#1}}
\providecommand{\bibinfo}[2]{#2}

\bibitemdeclare{inproceedings}{AbrCoe:CatSemQuant:2004}
\bibitem{AbrCoe:CatSemQuant:2004}
\bibinfo{author}{S.~\surnamestart Abramsky\surnameend} \&
  \bibinfo{author}{B.~\surnamestart Coecke\surnameend} (\bibinfo{year}{2004}):
  \emph{\bibinfo{title}{A categorical semantics of quantum protocols}}.
\newblock In: {\sl \bibinfo{booktitle}{Proceedings of the 19th Annual IEEE
  Symposium on Logic in Computer Science: LICS 2004}}, \bibinfo{publisher}{IEEE
  Computer Society}, pp. \bibinfo{pages}{415--425},
  \doi{10.1109/LICS.2004.1319636}.

\bibitemdeclare{inproceedings}{Backens:2015aa}
\bibitem{Backens:2015aa}
\bibinfo{author}{Miriam \surnamestart Backens\surnameend}
  (\bibinfo{year}{2015}): \emph{\bibinfo{title}{Making the stabilizer
  ZX-calculus complete for scalars}}.
\newblock In \bibinfo{editor}{Chris \surnamestart Heunen\surnameend},
  \bibinfo{editor}{Peter \surnamestart Selinger\surnameend} \&
  \bibinfo{editor}{Jamie \surnamestart Vicary\surnameend}, editors: {\sl
  \bibinfo{booktitle}{Proceedings of the 12th International Workshop on Quantum
  Physics and Logic (QPL 2015)}}, {\sl \bibinfo{series}{Electronic Proceedings
  in Theoretical Computer Science}} \bibinfo{volume}{195}, pp.
  \bibinfo{pages}{17--32}, \doi{10.4204/EPTCS.195.2}.

\bibitemdeclare{article}{balsam2012kitaev}
\bibitem{balsam2012kitaev}
\bibinfo{author}{Benjamin \surnamestart Balsam\surnameend} \&
  \bibinfo{author}{Alexander \surnamestart Kirillov~Jr\surnameend}
  (\bibinfo{year}{2012}): \emph{\bibinfo{title}{Kitaev's lattice model and
  {T}uraev-{V}iro {TQFT}s}}.
\newblock {\sl \bibinfo{journal}{ArXiv.org}}.
\newblock \urlprefix\url{https://arxiv.org/abs/1206.2308}.

\bibitemdeclare{article}{Beaudrap2017The-ZX-calculus}
\bibitem{Beaudrap2017The-ZX-calculus}
\bibinfo{author}{Niel \surnamestart de~Beaudrap\surnameend} \&
  \bibinfo{author}{Dominic \surnamestart Horsman\surnameend}
  (\bibinfo{year}{2020}): \emph{\bibinfo{title}{The ZX calculus is a language
  for surface code lattice surgery}}.
\newblock {\sl \bibinfo{journal}{{Quantum}}} \bibinfo{volume}{4}, p.
  \bibinfo{pages}{218}, \doi{10.22331/q-2020-01-09-218}.

\bibitemdeclare{article}{bespalov2000integrals}
\bibitem{bespalov2000integrals}
\bibinfo{author}{Yuri \surnamestart Bespalov\surnameend},
  \bibinfo{author}{Thomas \surnamestart Kerler\surnameend},
  \bibinfo{author}{Volodymyr \surnamestart Lyubashenko\surnameend} \&
  \bibinfo{author}{Vladimir \surnamestart Turaev\surnameend}
  (\bibinfo{year}{2000}): \emph{\bibinfo{title}{Integrals for braided Hopf
  algebras}}.
\newblock {\sl \bibinfo{journal}{Journal of Pure and Applied Algebra}}
  \bibinfo{volume}{148}(\bibinfo{number}{2}), pp. \bibinfo{pages}{113--164},
  \doi{10.1016/S0022-4049(98)00169-8}.

\bibitemdeclare{article}{Bonchi:2019:DAL:3302515.3290338}
\bibitem{Bonchi:2019:DAL:3302515.3290338}
\bibinfo{author}{Filippo \surnamestart Bonchi\surnameend},
  \bibinfo{author}{Joshua \surnamestart Holland\surnameend},
  \bibinfo{author}{Robin \surnamestart Piedeleu\surnameend},
  \bibinfo{author}{Pawe\l \surnamestart Soboci\'{n}ski\surnameend} \&
  \bibinfo{author}{Fabio \surnamestart Zanasi\surnameend}
  (\bibinfo{year}{2019}): \emph{\bibinfo{title}{Diagrammatic Algebra: From
  Linear to Concurrent Systems}}.
\newblock {\sl \bibinfo{journal}{Proceedings of the ACM on Programming
  Languages}} \bibinfo{volume}{3}(\bibinfo{number}{POPL}), pp.
  \bibinfo{pages}{25:1--25:28}, \doi{10.1145/3290338}.

\bibitemdeclare{article}{Bonchi2014aJournal}
\bibitem{Bonchi2014aJournal}
\bibinfo{author}{Filippo \surnamestart Bonchi\surnameend},
  \bibinfo{author}{Pawe{\l} \surnamestart Soboci\'nski\surnameend} \&
  \bibinfo{author}{Fabio \surnamestart Zanasi\surnameend}
  (\bibinfo{year}{2017}): \emph{\bibinfo{title}{Interacting {H}opf Algebras}}.
\newblock {\sl \bibinfo{journal}{Journal of Pure and Applied Algebra}}
  \bibinfo{volume}{221}(\bibinfo{number}{1}), pp. \bibinfo{pages}{144--184},
  \doi{10.1016/j.jpaa.2016.06.002}.

\bibitemdeclare{article}{buerschaper2013hierarchy}
\bibitem{buerschaper2013hierarchy}
\bibinfo{author}{Oliver \surnamestart Buerschaper\surnameend},
  \bibinfo{author}{Juan~Mart{\'\i}n \surnamestart Mombelli\surnameend},
  \bibinfo{author}{Matthias \surnamestart Christandl\surnameend} \&
  \bibinfo{author}{Miguel \surnamestart Aguado\surnameend}
  (\bibinfo{year}{2013}): \emph{\bibinfo{title}{A hierarchy of topological
  tensor network states}}.
\newblock {\sl \bibinfo{journal}{Journal of Mathematical Physics}}
  \bibinfo{volume}{54}(\bibinfo{number}{1}), p. \bibinfo{pages}{012201},
  \doi{10.1063/1.4773316}.

\bibitemdeclare{article}{Carbonia:1987aa}
\bibitem{Carbonia:1987aa}
\bibinfo{author}{A.~\surnamestart Carboni\surnameend} \&
  \bibinfo{author}{R.F.C. \surnamestart Walters\surnameend}
  (\bibinfo{year}{1987}): \emph{\bibinfo{title}{Cartesian bicategories I}}.
\newblock {\sl \bibinfo{journal}{Journal of Pure and Applied Algebra}}
  \bibinfo{volume}{49}(\bibinfo{number}{1-2}),
  \doi{10.1016/0022-4049(87)90121-6}.

\bibitemdeclare{article}{chen2000quantum}
\bibitem{chen2000quantum}
\bibinfo{author}{Hui-Xiang \surnamestart Chen\surnameend}
  (\bibinfo{year}{2000}): \emph{\bibinfo{title}{Quantum doubles in monoidal
  categories}}.
\newblock {\sl \bibinfo{journal}{Communications in Algebra}}
  \bibinfo{volume}{28}(\bibinfo{number}{5}), pp. \bibinfo{pages}{2303--2328},
  \doi{10.1080/00927870008826961}.

\bibitemdeclare{article}{PavlovicD:MSCS08}
\bibitem{PavlovicD:MSCS08}
\bibinfo{author}{B.~\surnamestart Coecke\surnameend},
  \bibinfo{author}{D.~\surnamestart Pavlovic\surnameend} \&
  \bibinfo{author}{J.~\surnamestart Vicary\surnameend} (\bibinfo{year}{2013}):
  \emph{\bibinfo{title}{A new description of orthogonal bases}}.
\newblock {\sl \bibinfo{journal}{Math. Structures in Comp. Sci.}}
  \bibinfo{volume}{23}(\bibinfo{number}{3}), pp. \bibinfo{pages}{555--567},
  \doi{10.1017/S0960129512000047}.

\bibitemdeclare{article}{Coecke:2009aa}
\bibitem{Coecke:2009aa}
\bibinfo{author}{Bob \surnamestart Coecke\surnameend} \& \bibinfo{author}{Ross
  \surnamestart Duncan\surnameend} (\bibinfo{year}{2011}):
  \emph{\bibinfo{title}{Interacting Quantum Observables: Categorical Algebra
  and Diagrammatics}}.
\newblock {\sl \bibinfo{journal}{New J. Phys}}
  \bibinfo{volume}{13}(\bibinfo{number}{043016}),
  \doi{10.1088/1367-2630/13/4/043016}.
\newblock \urlprefix\url{http://iopscience.iop.org/1367-2630/13/4/043016/}.

\bibitemdeclare{book}{Coecke2017Picturing-Quant}
\bibitem{Coecke2017Picturing-Quant}
\bibinfo{author}{Bob \surnamestart Coecke\surnameend} \& \bibinfo{author}{Aleks
  \surnamestart Kissinger\surnameend} (\bibinfo{year}{2017}):
  \emph{\bibinfo{title}{Picturing Quantum Processes: A First Course in Quantum
  Theory and Diagrammatic Reasoning}}.
\newblock \bibinfo{publisher}{Cambridge University Press},
  \doi{10.1017/9781316219317}.

\bibitemdeclare{incollection}{Doi2000Bi-Frobenius-al}
\bibitem{Doi2000Bi-Frobenius-al}
\bibinfo{author}{Yukio \surnamestart Doi\surnameend} \&
  \bibinfo{author}{Mitsuhiro \surnamestart Takeuchi\surnameend}
  (\bibinfo{year}{2000}): \emph{\bibinfo{title}{Bi-Frobenius algebras}}.
\newblock In \bibinfo{editor}{Nicol{\'a}s \surnamestart
  Andruskiewitsch\surnameend}, \bibinfo{editor}{Walter Ricardo~Ferrer
  \surnamestart Santos\surnameend} \& \bibinfo{editor}{Hans-J{\"u}rgen
  \surnamestart Schneider\surnameend}, editors: {\sl \bibinfo{booktitle}{New
  trends in {H}opf algebra theory}}, {\sl \bibinfo{series}{Contemporary
  Mathematics}} \bibinfo{volume}{267}, \bibinfo{publisher}{American
  Mathematical Society}, pp. \bibinfo{pages}{67--98},
  \doi{10.1090/conm/267/04265}.

\bibitemdeclare{article}{drinfeld1986quantum}
\bibitem{drinfeld1986quantum}
\bibinfo{author}{Vladimir~Gershonovich \surnamestart Drinfeld\surnameend}
  (\bibinfo{year}{1986}): \emph{\bibinfo{title}{Quantum groups}}.
\newblock {\sl \bibinfo{journal}{Zapiski Nauchnykh Seminarov POMI}}
  \bibinfo{volume}{155}, pp. \bibinfo{pages}{18--49}, \doi{10.1007/BF01247086}.

\bibitemdeclare{inproceedings}{Duncan2016Interacting-Fro}
\bibitem{Duncan2016Interacting-Fro}
\bibinfo{author}{Ross \surnamestart Duncan\surnameend} \&
  \bibinfo{author}{Kevin \surnamestart Dunne\surnameend}
  (\bibinfo{year}{2016}): \emph{\bibinfo{title}{Interacting {F}robenius
  Algebras are {H}opf}}.
\newblock In \bibinfo{editor}{Martin \surnamestart Grohe\surnameend},
  \bibinfo{editor}{Eric \surnamestart Koskinen\surnameend} \&
  \bibinfo{editor}{Natarajan \surnamestart Shankar\surnameend}, editors: {\sl
  \bibinfo{booktitle}{{Proceedings of the 31st Annual ACM/IEEE Symposium on
  Logic in Computer Science, {LICS} '16, New York, NY, USA, July 5-8, 2016}}},
  \bibinfo{series}{{LICS '16}}, \bibinfo{publisher}{{ACM}}, pp.
  \bibinfo{pages}{535--544}, \doi{10.1145/2933575.2934550}.

\bibitemdeclare{incollection}{Bertfried:2013aa}
\bibitem{Bertfried:2013aa}
\bibinfo{author}{Bertfried \surnamestart Fauser\surnameend}
  (\bibinfo{year}{2013}): \emph{\bibinfo{title}{Some Graphical Aspects of
  {F}robenius Algebras}}.
\newblock In \bibinfo{editor}{Chris \surnamestart Heunen\surnameend},
  \bibinfo{editor}{Mehrnoosh \surnamestart Sadrzadeh\surnameend} \&
  \bibinfo{editor}{Edward \surnamestart Grefenstette\surnameend}, editors: {\sl
  \bibinfo{booktitle}{Quantum Physics and Linguistics: A Compositional,
  Diagrammatic Discourse}}, \bibinfo{publisher}{Oxford},
  \doi{10.1093/acprof:oso/9780199646296.003.0002}.

\bibitemdeclare{inproceedings}{Gogioso2016Infinite-dimens}
\bibitem{Gogioso2016Infinite-dimens}
\bibinfo{author}{Stefano \surnamestart Gogioso\surnameend} \&
  \bibinfo{author}{Fabrizio \surnamestart Genovese\surnameend}
  (\bibinfo{year}{2016}): \emph{\bibinfo{title}{Infinite-dimensional
  Categorical Quantum Mechanics}}.
\newblock In: {\sl \bibinfo{booktitle}{Proceedings of QPL 2016}},
  \bibinfo{series}{EPTCS}, \doi{10.4204/EPTCS.236.4}.

\bibitemdeclare{article}{hasegawa2010bialgebras}
\bibitem{hasegawa2010bialgebras}
\bibinfo{author}{Masahito \surnamestart Hasegawa\surnameend}
  (\bibinfo{year}{2010}): \emph{\bibinfo{title}{Bialgebras in rel}}.
\newblock {\sl \bibinfo{journal}{Electronic Notes in Theoretical Computer
  Science}} \bibinfo{volume}{265}, pp. \bibinfo{pages}{337--350},
  \doi{10.1016/j.entcs.2010.08.020}.

\bibitemdeclare{inproceedings}{Jeandel2017A-Complete-Axio}
\bibitem{Jeandel2017A-Complete-Axio}
\bibinfo{author}{Emmanuel \surnamestart Jeandel\surnameend},
  \bibinfo{author}{Simon \surnamestart Perdrix\surnameend} \&
  \bibinfo{author}{Renaud \surnamestart Vilmart\surnameend}
  (\bibinfo{year}{2017}): \emph{\bibinfo{title}{A Complete Axiomatisation of
  the {ZX}-Calculus for {C}lifford+{T} Quantum Mechanics}}.
\newblock In: {\sl \bibinfo{booktitle}{LICS '18- Proceedings of the 33rd Annual
  ACM/IEEE Symposium on Logic in Computer Science}},
  \bibinfo{volume}{arXiv:1705.11151}, \bibinfo{publisher}{ACM},
  \doi{10.1145/3209108.3209131}.

\bibitemdeclare{book}{kassel2012quantum}
\bibitem{kassel2012quantum}
\bibinfo{author}{Christian \surnamestart Kassel\surnameend}
  (\bibinfo{year}{2012}): \emph{\bibinfo{title}{Quantum groups}}.
\newblock \bibinfo{volume}{155}, \bibinfo{publisher}{Springer Science \&
  Business Media}, \doi{10.4171/047}.

\bibitemdeclare{article}{KelLap:comcl:1980}
\bibitem{KelLap:comcl:1980}
\bibinfo{author}{G.M. \surnamestart Kelly\surnameend} \& \bibinfo{author}{M.L.
  \surnamestart Laplaza\surnameend} (\bibinfo{year}{1980}):
  \emph{\bibinfo{title}{Coherence for Compact Closed Categories}}.
\newblock {\sl \bibinfo{journal}{Journal of Pure and Applied Algebra}}
  \bibinfo{volume}{19}, pp. \bibinfo{pages}{193--213},
  \doi{10.1016/0022-4049(80)90101-2}.

\bibitemdeclare{book}{Kock:TQFT:2003}
\bibitem{Kock:TQFT:2003}
\bibinfo{author}{J.~\surnamestart Kock\surnameend} (\bibinfo{year}{2003}):
  \emph{\bibinfo{title}{Frobenius Algebras and 2-D Topological Quantum Field
  Theories}}.
\newblock \bibinfo{publisher}{Cambridge University Press},
  \doi{10.1017/cbo9780511615443}.

\bibitemdeclare{article}{KOPPINEN1996256}
\bibitem{KOPPINEN1996256}
\bibinfo{author}{M.~\surnamestart Koppinen\surnameend} (\bibinfo{year}{1996}):
  \emph{\bibinfo{title}{On Algebras with Two Multiplications, Including Hopf
  Algebras and Bose--Mesner Algebras}}.
\newblock {\sl \bibinfo{journal}{Journal of Algebra}}
  \bibinfo{volume}{182}(\bibinfo{number}{1}), pp. \bibinfo{pages}{256 -- 273},
  \doi{10.1006/jabr.1996.0170}.
\newblock
  \urlprefix\url{http://www.sciencedirect.com/science/article/pii/S0021869396901702}.

\bibitemdeclare{article}{Lack:2004sf}
\bibitem{Lack:2004sf}
\bibinfo{author}{Stephen \surnamestart Lack\surnameend} (\bibinfo{year}{2004}):
  \emph{\bibinfo{title}{Composing {PROP}s}}.
\newblock {\sl \bibinfo{journal}{Theory and Applications of Categories}}
  \bibinfo{volume}{13}(\bibinfo{number}{9}), pp. \bibinfo{pages}{147--163}.

\bibitemdeclare{article}{Larson1969An-Associative-}
\bibitem{Larson1969An-Associative-}
\bibinfo{author}{Richard~Gustavus \surnamestart Larson\surnameend} \&
  \bibinfo{author}{Moss~Eisenberg \surnamestart Sweedler\surnameend}
  (\bibinfo{year}{1969}): \emph{\bibinfo{title}{An Associative Orthogonal
  Bilinear Form for {H}opf Algebras}}.
\newblock {\sl \bibinfo{journal}{American Journal of Mathematics}}
  \bibinfo{volume}{91}(\bibinfo{number}{1}), pp. \bibinfo{pages}{75--94},
  \doi{10.2307/2373270}.
\newblock \urlprefix\url{https://www.jstor.org/stable/2373270}.

\bibitemdeclare{article}{Meusburger2017}
\bibitem{Meusburger2017}
\bibinfo{author}{Catherine \surnamestart Meusburger\surnameend}
  (\bibinfo{year}{2017}): \emph{\bibinfo{title}{Kitaev Lattice Models as a
  {H}opf Algebra Gauge Theory}}.
\newblock {\sl \bibinfo{journal}{Communications in Mathematical Physics}}
  \bibinfo{volume}{353}(\bibinfo{number}{1}), pp. \bibinfo{pages}{413--468},
  \doi{10.1007/s00220-017-2860-7}.

\bibitemdeclare{article}{PAREIGIS1971588}
\bibitem{PAREIGIS1971588}
\bibinfo{author}{Bodo \surnamestart Pareigis\surnameend}
  (\bibinfo{year}{1971}): \emph{\bibinfo{title}{When Hopf algebras are
  Frobenius algebras}}.
\newblock {\sl \bibinfo{journal}{Journal of Algebra}}
  \bibinfo{volume}{18}(\bibinfo{number}{4}), pp. \bibinfo{pages}{588 -- 596},
  \doi{10.1016/0021-8693(71)90141-4}.
\newblock
  \urlprefix\url{http://www.sciencedirect.com/science/article/pii/0021869371901414}.

\bibitemdeclare{inproceedings}{Selinger:2010:aa}
\bibitem{Selinger:2010:aa}
\bibinfo{author}{P.~\surnamestart Selinger\surnameend} (\bibinfo{year}{2010}):
  \emph{\bibinfo{title}{Autonomous categories in which ${A} \cong {A}^*$}}.
\newblock In \bibinfo{editor}{B.~\surnamestart Coecke\surnameend},
  \bibinfo{editor}{P.~\surnamestart Panangaden\surnameend} \&
  \bibinfo{editor}{P.~\surnamestart Selinger\surnameend}, editors: {\sl
  \bibinfo{booktitle}{Proceedings of 7th Workshop on Quantum Physics and Logic
  (QPL 2010)}}.
\newblock
  \urlprefix\url{http://www.mscs.dal.ca/~selinger/papers/halftwist-2up.pdf}.

\bibitemdeclare{incollection}{Selinger:2009aa}
\bibitem{Selinger:2009aa}
\bibinfo{author}{Peter \surnamestart Selinger\surnameend}
  (\bibinfo{year}{2011}): \emph{\bibinfo{title}{A survey of graphical languages
  for monoidal categories}}.
\newblock In \bibinfo{editor}{Bob \surnamestart Coecke\surnameend}, editor:
  {\sl \bibinfo{booktitle}{New structures for physics}}, {\sl
  \bibinfo{series}{Lecture Notes in Physics}} \bibinfo{volume}{813},
  \bibinfo{publisher}{Springer}, pp. \bibinfo{pages}{289--355},
  \doi{10.1007/978-3-642-12821-9\_4}.

\bibitemdeclare{book}{street2007quantum}
\bibitem{street2007quantum}
\bibinfo{author}{R.~\surnamestart Street\surnameend} (\bibinfo{year}{2007}):
  \emph{\bibinfo{title}{Quantum Groups: A Path to Current Algebra}}.
\newblock \bibinfo{series}{Australian Mathematical Society Lecture Series},
  \bibinfo{publisher}{Cambridge University Press},
  \doi{10.1017/CBO9780511618505}.

\bibitemdeclare{book}{Sweedler1969Hopf-Algebras}
\bibitem{Sweedler1969Hopf-Algebras}
\bibinfo{author}{Moss~E. \surnamestart Sweedler\surnameend}
  (\bibinfo{year}{1969}): \emph{\bibinfo{title}{Hopf Algebras}}.
\newblock \bibinfo{publisher}{W. A. Benjamin Inc.}

\bibitemdeclare{article}{Taft2631}
\bibitem{Taft2631}
\bibinfo{author}{Earl~J. \surnamestart Taft\surnameend} (\bibinfo{year}{1971}):
  \emph{\bibinfo{title}{The Order of the Antipode of Finite-dimensional Hopf
  Algebra}}.
\newblock {\sl \bibinfo{journal}{Proceedings of the National Academy of
  Sciences}} \bibinfo{volume}{68}(\bibinfo{number}{11}), pp.
  \bibinfo{pages}{2631--2633}, \doi{10.1073/pnas.68.11.2631}.
\newblock \urlprefix\url{https://www.pnas.org/content/68/11/2631}.

\bibitemdeclare{article}{takeuchi1999finite}
\bibitem{takeuchi1999finite}
\bibinfo{author}{Mitsuhiro \surnamestart Takeuchi\surnameend}
  (\bibinfo{year}{1999}): \emph{\bibinfo{title}{Finite Hopf algebras in braided
  tensor categories}}.
\newblock {\sl \bibinfo{journal}{Journal of Pure and Applied Algebra}}
  \bibinfo{volume}{138}(\bibinfo{number}{1}), pp. \bibinfo{pages}{59--82},
  \doi{10.1016/s0022-4049(97)00207-7}.

\end{thebibliography}
%\bibliography{all}
%\end{document}

\clearpage
\normalsize

%%%%%%%%%%%%%%%%%%%%%%%%%%%%%%%%%%%%%%%%%%%%%%%%%%%%%%%
%%% APPENDIX
%%%%%%%%%%%%%%%%%%%%%%%%%%%%%%%%%%%%%%%%%%%%%%%%%%%%%%%

%%%%%%%%%%%%%%%%%%%%%%%%%%%%%%%%%%%%%%%%%%%%%%%%%%%%%%%
%%% APPENDIX
%%%%%%%%%%%%%%%%%%%%%%%%%%%%%%%%%%%%%%%%%%%%%%%%%%%%%%%

\appendix

\section{Proofs omitted from the main body of the paper}
\label{sec:proofs-omitted-from}

\begin{prevlemma}{theorem:invertible-antipode}
  Let $(H, \rint, \gcoint)$ be an integral Hopf algebra. Then the
  following map is the inverse of the antipode.
\[
\begin{tikzpicture}
	\begin{pgfonlayer}{nodelayer}
		\node [style=red vertex] (0) at (-0.25, 0.25) {};
		\node [style=green vertex] (3) at (0.25, -0.25) {};
		\node [style=none] (4) at (-0.5, -0.25) {};
		\node [style=none] (5) at (0.5, 0.25) {};
		\node [style=integral] (6) at (-0.25, 0.5) {};
		\node [style=cointegral] (7) at (0.25, -0.5) {};
		\node [style=none] (9) at (-1, 0) {:=};
		\node [style=none] (10) at (-1.75, 0.5) {};
		\node [style=none] (11) at (-1.75, -0.5) {};
		\node [style=antipode] (12) at (-1.75, 0) {};
		\node [style=none] (13) at (-1.5, 0.25) {\scriptsize -1};
	\end{pgfonlayer}
	\begin{pgfonlayer}{edgelayer}
		\draw [style=directed edge] (6) to (0);
		\draw [style=directed edge] (3) to (7);
		\draw [style=directed edge, bend right=15] (0) to (4.center);
		\draw [style=directed edge, bend left=15] (5.center) to (3);
		\draw [style=directed edge, in=135, out=-30] (0) to (3);
		\draw (10.center) to (12);
		\draw (12) to (11.center);
	\end{pgfonlayer}
\end{tikzpicture}
\]
\end{prevlemma}

\begin{proof}
  From the definition of $\antipode^{-1}$, we see that
  \[
  \begin{tikzpicture}
	\begin{pgfonlayer}{nodelayer}
		\node [style=red vertex] (0) at (-2.25, 0.5) {};
		\node [style=red vertex] (1) at (-1.5, 0.5) {};
		\node [style=green vertex] (2) at (-2.25, -0.25) {};
		\node [style=green vertex] (3) at (-1.5, -0.25) {};
		\node [style=none] (4) at (-1.5, 1) {};
		\node [style=none] (5) at (-2.25, -0.75) {};
		\node [style=integral] (6) at (-2.25, 0.75) {};
		\node [style=cointegral] (7) at (-1.5, -0.5) {};
		\node [style=none] (8) at (-0.75, 0) {=};
		\node [style=none] (9) at (0.25, 0.75) {};
		\node [style=none] (10) at (-0.25, -0.75) {};
		\node [style=green vertex] (11) at (0, 0.25) {};
		\node [style=red vertex] (12) at (0, -0.25) {};
		\node [style=integral] (13) at (-0.25, 0.75) {};
		\node [style=cointegral] (14) at (0.25, -0.75) {};
		\node [style=none] (15) at (0.75, 0) {=};
		\node [style=none] (16) at (1.75, 1) {};
		\node [style=none] (17) at (1.75, -0.75) {};
		\node [style=integral] (20) at (1.5, 0.25) {};
		\node [style=cointegral] (21) at (1.5, 0) {};
		\node [style=green vertex] (22) at (1.75, -0.25) {};
		\node [style=red vertex] (23) at (1.75, 0.5) {};
		\node [style=none] (24) at (2.5, 0) {=};
		\node [style=none] (25) at (3.25, 1) {};
		\node [style=none] (26) at (3.25, -0.75) {};
		\node [style=green vertex] (29) at (3.25, -0.25) {};
		\node [style=red vertex] (30) at (3.25, 0.5) {};
		\node [style=none] (40) at (-2.75, 0) {=};
		\node [style=none] (41) at (-3.5, 0) {\tiny -1};
		\node [style=none] (44) at (-3.5, 1.5) {};
		\node [style=red vertex] (46) at (-3.5, 1) {};
		\node [style=none] (48) at (-3.5, -1.25) {};
		\node [style=green vertex] (49) at (-3.5, -0.75) {};
		\node [style=antipode] (50) at (-3.75, -0.25) {};
		\node [style=none] (51) at (-3.25, -0.25) {};
		\node [style=none] (52) at (-3.25, 0.5) {};
		\node [style=none] (54) at (-3.75, 0) {};
		\node [style=none] (55) at (-3.75, 0.5) {};
		\node [style=none] (56) at (-3.25, 0) {};
	\end{pgfonlayer}
	\begin{pgfonlayer}{edgelayer}
		\draw [bend right=15] (0) to (2);
		\draw [bend left=15] (1) to (3);
		\draw (0) to (3);
		\draw (1) to (2);
		\draw (4.center) to (1);
		\draw (5.center) to (2);
		\draw (6) to (0);
		\draw (3) to (7);
		\draw [bend left=15] (9.center) to (11);
		\draw [bend right=15] (13) to (11);
		\draw (11) to (12);
		\draw [bend right=15] (14) to (12);
		\draw [bend left=15] (10.center) to (12);
		\draw (20) to (21);
		\draw (22) to (17.center);
		\draw (23) to (16.center);
		\draw (29) to (26.center);
		\draw (30) to (25.center);
		\draw (44.center) to (46);
		\draw (48.center) to (49);
		\draw [in=-90, out=120] (49) to (50);
		\draw [in=-90, out=60] (49) to (51.center);
		\draw [in=90, out=-90] (52.center) to (54.center);
		\draw [in=90, out=-90] (55.center) to (56.center);
		\draw (54.center) to (50);
		\draw (56.center) to (51.center);
		\draw [in=90, out=-30] (46) to (52.center);
		\draw [in=90, out=-150] (46) to (55.center);
	\end{pgfonlayer}
\end{tikzpicture}
  \]
  This implies that $H^\sigma$ is a Hopf algebra with $\antipode^{-1}$
  as the antipode, and it follows from Proposition
  \ref{prop:antipode-props} that the antipode of $H^\sigma$ is the
  inverse of the antipode of $H$. However, for the sake of clarity we
  will replicate the proof. We show that
  $\antipode^{-1} \circ \antipode = 1$ as follows
  \[
  \begin{tikzpicture}[label distance=-5pt]
	\begin{pgfonlayer}{nodelayer}
		\node [style=none] (63) at (2.25, -1.25) {};
		\node [style=none] (65) at (1.75, -1.25) {};
		\node [style=none] (15) at (-1, 0) {=};
		\node [style=none] (22) at (-2.5, 0) {=};
		\node [style=none] (23) at (-3.25, 1) {};
		\node [style=none] (24) at (-3.25, -0.75) {};
		\node [style=green vertex] (25) at (-3.25, -0.25) {};
		\node [style=red vertex] (26) at (-3.25, 0.5) {};
		\node [style=none] (27) at (1, 0) {=};
		\node [style=none] (29) at (0, 1.5) {};
		\node [style=red vertex] (30) at (0, 1) {};
		\node [style=antipode, label={240:\tiny -1}] (33) at (-0.25, -0.25) {};
		\node [style=none] (34) at (0.25, -0.25) {};
		\node [style=none] (35) at (0.25, 0.5) {};
		\node [style=none] (36) at (-0.25, 0) {};
		\node [style=none] (37) at (-0.25, 0.5) {};
		\node [style=none] (38) at (0.25, 0) {};
		\node [style=antipode] (39) at (-1.75, -0.75) {};
		\node [style=none] (40) at (-1.75, 1) {};
		\node [style=none] (41) at (-1.75, -0.75) {};
		\node [style=green vertex] (42) at (-1.75, -0.25) {};
		\node [style=red vertex] (43) at (-1.75, 0.5) {};
		\node [style=none] (44) at (-1.75, -1) {};
		\node [style=antipode] (45) at (0, -1.25) {};
		\node [style=none] (46) at (0, -1.25) {};
		\node [style=green vertex] (47) at (0, -0.75) {};
		\node [style=none] (48) at (0, -1.5) {};
		\node [style=none] (50) at (2, 1.75) {};
		\node [style=red vertex] (51) at (2, 1.25) {};
		\node [style=none] (53) at (2.25, 0) {};
		\node [style=none] (54) at (2.25, 0.75) {};
		\node [style=none] (55) at (1.75, 0.25) {};
		\node [style=none] (56) at (1.75, 0.75) {};
		\node [style=none] (57) at (2.25, 0.25) {};
		\node [style=none] (59) at (2, -2) {};
		\node [style=green vertex] (60) at (2, -1.5) {};
		\node [style=none] (67) at (2.25, -0.75) {};
		\node [style=none] (68) at (1.75, -0.25) {};
		\node [style=none] (69) at (1.75, -0.75) {};
		\node [style=none] (70) at (2.25, -0.25) {};
		\node [style=antipode] (71) at (1.75, -1) {};
		\node [style=antipode] (72) at (2.25, -1) {};
		\node [style=antipode, label={240:\tiny -1}] (73) at (1.75, 0) {};
		\node [style=none] (74) at (3, 0) {=};
		\node [style=none] (75) at (4, -0.75) {};
		\node [style=none] (76) at (3.5, -0.75) {};
		\node [style=none] (77) at (3.75, 1.25) {};
		\node [style=red vertex] (78) at (3.75, 0.75) {};
		\node [style=none] (80) at (4, 0.25) {};
		\node [style=none] (82) at (3.5, 0.25) {};
		\node [style=none] (84) at (3.75, -1.5) {};
		\node [style=green vertex] (85) at (3.75, -1) {};
		\node [style=none] (86) at (4, -0.25) {};
		\node [style=none] (88) at (3.5, -0.25) {};
		\node [style=antipode] (90) at (3.5, -0.5) {};
		\node [style=antipode] (91) at (4, -0.5) {};
		\node [style=antipode, label={240:\tiny -1}] (92) at (4, 0) {};
	\end{pgfonlayer}
	\begin{pgfonlayer}{edgelayer}
		\draw (25) to (24.center);
		\draw (26) to (23.center);
		\draw (29.center) to (30);
		\draw [in=90, out=-90] (35.center) to (36.center);
		\draw [in=90, out=-90] (37.center) to (38.center);
		\draw (36.center) to (33);
		\draw (38.center) to (34.center);
		\draw [in=90, out=-30] (30) to (35.center);
		\draw [in=90, out=-150] (30) to (37.center);
		\draw (42) to (41.center);
		\draw (43) to (40.center);
		\draw (41.center) to (39);
		\draw (39) to (44.center);
		\draw (47) to (46.center);
		\draw (46.center) to (45);
		\draw (45) to (48.center);
		\draw [in=-90, out=120] (47) to (33);
		\draw [in=-90, out=60] (47) to (34.center);
		\draw (50.center) to (51);
		\draw [in=90, out=-90] (54.center) to (55.center);
		\draw [in=90, out=-90] (56.center) to (57.center);
		\draw (57.center) to (53.center);
		\draw [in=90, out=-30] (51) to (54.center);
		\draw [in=90, out=-150] (51) to (56.center);
		\draw (60) to (59.center);
		\draw [in=-90, out=30] (60) to (63.center);
		\draw [in=-90, out=150] (60) to (65.center);
		\draw [in=-90, out=90] (67.center) to (68.center);
		\draw [in=-90, out=90] (69.center) to (70.center);
		\draw (65.center) to (71);
		\draw (63.center) to (72);
		\draw (53.center) to (70.center);
		\draw (67.center) to (72);
		\draw (69.center) to (71);
		\draw (55.center) to (73);
		\draw (73) to (68.center);
		\draw (77.center) to (78);
		\draw [in=90, out=-30] (78) to (80.center);
		\draw [in=90, out=-150] (78) to (82.center);
		\draw (85) to (84.center);
		\draw [in=-90, out=30] (85) to (75.center);
		\draw [in=-90, out=150] (85) to (76.center);
		\draw (76.center) to (90);
		\draw (75.center) to (91);
		\draw (86.center) to (91);
		\draw (88.center) to (90);
		\draw (80.center) to (92);
		\draw (86.center) to (92);
		\draw (88.center) to (82.center);
	\end{pgfonlayer}
\end{tikzpicture}
  \]
  % This implies that $\antipode \circ \antipode^{-1}$ is the
  % convolution inverse of $\antipode$. By definition, the identity map
  % $1_H$ is the convolution inverse of the antipode, and by uniqueness
  % of inverses, we get that $\antipode \circ \antipode^{-1} = 1$.

  It follows from this that 
  \[
  \begin{tikzpicture}
	\begin{pgfonlayer}{nodelayer}
		\node [style=none] (14) at (-2, 0.5) {};
		\node [style=none] (0) at (-3, 1) {};
		\node [style=none] (1) at (-3, -0.75) {};
		\node [style=none] (2) at (-2.5, 0.25) {=};
		\node [style=none] (3) at (-1.75, 1.5) {};
		\node [style=none] (4) at (-1.75, -1.25) {};
		\node [style=red vertex] (5) at (-1.75, 1) {};
		\node [style=red vertex] (6) at (-1.5, 0.5) {};
		\node [style=green vertex] (8) at (-1.5, -0.25) {};
		\node [style=green vertex] (9) at (-1.75, -0.75) {};
		\node [style=none] (11) at (-2, -0.25) {};
		\node [style=none] (15) at (-0.5, 0.75) {};
		\node [style=none] (16) at (0, 1.5) {};
		\node [style=none] (17) at (0, -1.25) {};
		\node [style=none] (22) at (-0.5, -0.5) {};
		\node [style=none] (23) at (2, -0.5) {};
		\node [style=none] (25) at (1.75, 1.5) {};
		\node [style=none] (27) at (2, 0.75) {};
		\node [style=none] (29) at (1.75, -1.25) {};
		\node [style=none] (31) at (2, 0) {};
		\node [style=antipode] (34) at (2, -0.25) {};
		\node [style=antipode, label={240:\tiny -1}] (35) at (2, 0.5) {};
		\node [style=none] (36) at (0.5, -0.5) {};
		\node [style=none] (37) at (0, -0.5) {};
		\node [style=red vertex] (38) at (0, 1) {};
		\node [style=none] (39) at (0.5, 0.75) {};
		\node [style=none] (40) at (0, 0.5) {};
		\node [style=green vertex] (41) at (0, -0.75) {};
		\node [style=none] (42) at (0.5, 0) {};
		\node [style=none] (43) at (0, 0) {};
		\node [style=antipode] (44) at (0, -0.25) {};
		\node [style=antipode] (45) at (0.5, -0.25) {};
		\node [style=antipode, label={240:\tiny -1}] (46) at (0.5, 0.5) {};
		\node [style=none] (58) at (-1, 0.25) {=};
		\node [style=red vertex] (59) at (1.75, 1) {};
		\node [style=red vertex] (60) at (1.5, 0.5) {};
		\node [style=green vertex] (61) at (1.5, -0.25) {};
		\node [style=green vertex] (62) at (1.75, -0.75) {};
		\node [style=none] (63) at (1, 0.25) {=};
		\node [style=none] (64) at (3.25, -0.75) {};
		\node [style=none] (65) at (3.25, 1) {};
		\node [style=none] (66) at (3.25, 0) {};
		\node [style=antipode] (67) at (3.25, -0.25) {};
		\node [style=antipode, label={240:\tiny -1}] (68) at (3.25, 0.5) {};
		\node [style=none] (69) at (2.5, 0.25) {=};
	\end{pgfonlayer}
	\begin{pgfonlayer}{edgelayer}
		\draw (0.center) to (1.center);
		\draw (5) to (3.center);
		\draw [in=-90, out=150] (9) to (11.center);
		\draw (4.center) to (9);
		\draw [in=-90, out=30] (9) to (8);
		\draw [in=90, out=-150] (5) to (14.center);
		\draw [in=90, out=-30] (5) to (6);
		\draw (14.center) to (11.center);
		\draw (15.center) to (22.center);
		\draw (23.center) to (34);
		\draw (31.center) to (34);
		\draw (27.center) to (35);
		\draw (31.center) to (35);
		\draw (38) to (40.center);
		\draw [in=-90, out=0] (41) to (36.center);
		\draw (41) to (37.center);
		\draw (37.center) to (44);
		\draw (36.center) to (45);
		\draw (42.center) to (45);
		\draw (43.center) to (44);
		\draw (39.center) to (46);
		\draw (42.center) to (46);
		\draw (43.center) to (40.center);
		\draw [in=90, out=0] (38) to (39.center);
		\draw [in=-90, out=180] (41) to (22.center);
		\draw (17.center) to (41);
		\draw (38) to (16.center);
		\draw [in=90, out=-180] (38) to (15.center);
		\draw [in=-90, out=150] (62) to (61);
		\draw [in=90, out=-150] (59) to (60);
		\draw (25.center) to (59);
		\draw [in=90, out=-30] (59) to (27.center);
		\draw (62) to (29.center);
		\draw [in=-90, out=30] (62) to (23.center);
		\draw (64.center) to (67);
		\draw (66.center) to (67);
		\draw (65.center) to (68);
		\draw (66.center) to (68);
	\end{pgfonlayer}
\end{tikzpicture}
  \]

  By a similar argument, $\antipode^{-1} \circ \antipode = 1$.

  %% This implies that $H^\sigma$ is a Hopf algebra. Recall that by
  %% Proposition \ref{prop:antipode-props}, when $H^\sigma$ is a Hopf
  %% algebra, then the antipode of $H^\sigma$ is the inverse of the
  %% antipode of $H$. Thus, $\antipode^{-1} \circ \antipode = 1$
\end{proof}

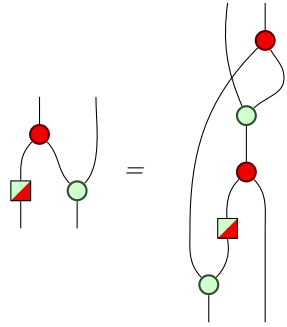
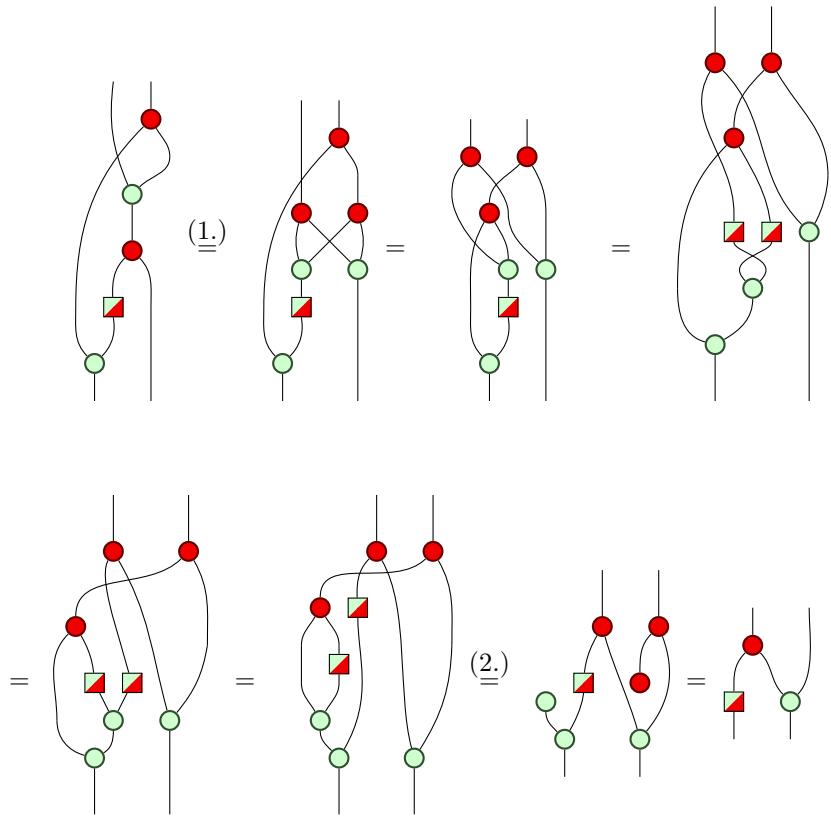
\begin{lemma}\label{lemma:alternate-bialgebra}
\[
\inlfropf{alt-bialgebra-rule}
\]
\end{lemma}
\begin{proof}
  Observe that
\[
\inlfropf{alt-bialgebra-proof-1}
\]
where $(1.)$ comes from the bialgebra rule and $(2.)$ comes from the
Hopf law.
\end{proof}

\begin{prevlemma}{lem:equivalent-to-nondegenerate}
    Let $(H, \rint, \gcoint)$ be an integral Hopf algebra. $\gamma$ is
    a right inverse for $\beta$ if and only if
    \[
    \inltf{hopf-is-frob/equivalent-to-nondeg}\]
    
  \end{prevlemma}
  \begin{proof}
    Suppose that $H$ is nondegenerate. Then
    \[
    \hopfisfrob{eq-to-nondeg-proof-4}
    \]
    
    Consider the converse. Then we may characterise the unit as follows:
    \[
    \hopfisfrob{eq-to-nondeg-proof-1}
    \]
  This implies that
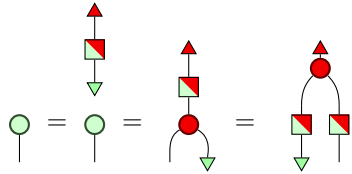
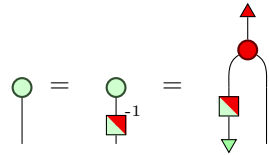
    \[
    \hopfisfrob{eq-to-nondeg-proof-2}
    \]
    This then allows us to show the following, where (1.) comes from
    Lemma \ref{lemma:alternate-bialgebra}, and $(2.)$ is due to the
    definition of cointegrals.
    \[
    \hopfisfrob{eq-to-nondeg-proof-3}
    \]
    Hence, $H$ is nondegenerate, and we have our result.

  \end{proof}

\begin{prevlemma}{lem:nondegenerate-implies-preHF}
  Let $(H, \rint, \gcoint)$ be an integral Hopf
  algebra. $H$ is nondegenerate if and only if $\beta$ is a Frobenius
  form for $(H,\gmult,\gunit)$, or equivalently, if and only if
  $\gamma'$ is a Frobenius form for $(H,\rcomult,\rcounit)$. Hence, if
  $H$ is nondegenerate, then $H$ admits a pre-HF algebra structure.
\end{prevlemma}
\begin{proof}
  If $H$ is nondegenerate, then the conditions of Definition~\ref{def:frob-alt} are
  satisfied.

  Conversely, suppose that $\beta$ is a Frobenius form;
  then there exists some $\bar\beta$ such that 
  
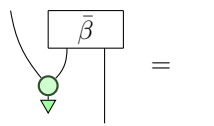
\[
  \begin{tikzpicture}
	\begin{pgfonlayer}{nodelayer}
		\node [style=none] (0) at (-0.75, 0.25) {};
		\node [style=none] (1) at (-1.25, 0.25) {};
		\node [style=none] (2) at (-1.5, 0.25) {};
		\node [style=none] (3) at (-0.5, 0.25) {};
		\node [style=none] (4) at (-0.5, 0.75) {};
		\node [style=none] (5) at (-1.5, 0.75) {};
		\node [style=none] (6) at (-1, 0.5) {$\bar\beta$};
		\node [style=green vertex] (8) at (-1.5, -0.25) {};
		\node [style=none] (10) at (-2, 0.75) {};
		\node [style=none] (11) at (-0.75, -0.75) {};
		\node [style=none] (12) at (0, 0) {=};
		\node [style=none] (13) at (0.75, 0.75) {};
		\node [style=none] (14) at (0.75, -0.75) {};
		\node [style=cointegral] (15) at (-1.5, -0.5) {};
	\end{pgfonlayer}
	\begin{pgfonlayer}{edgelayer}
		\draw (3.center) to (2.center);
		\draw (4.center) to (3.center);
		\draw (5.center) to (4.center);
		\draw (2.center) to (5.center);
		\draw [style=directed edge, in=45, out=-90] (1.center) to (8);
		\draw [style=directed edge] (0.center) to (11.center);
		\draw [style=directed edge, bend right=15] (10.center) to (8);
		\draw [style=directed edge] (13.center) to (14.center);
		\draw [style=directed edge] (8) to (15);
	\end{pgfonlayer}
\end{tikzpicture}
  \]
  Appealing to Lemma~\ref{theorem:invertible-antipode} we have
  
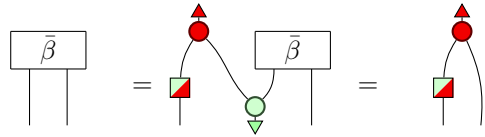
\[
  \begin{tikzpicture}
	\begin{pgfonlayer}{nodelayer}
		\node [style=none] (0) at (0.25, 0) {};
		\node [style=none] (1) at (-0.25, 0) {};
		\node [style=none] (2) at (-0.5, 0) {};
		\node [style=none] (3) at (0.5, 0) {};
		\node [style=none] (4) at (0.5, 0.5) {};
		\node [style=none] (5) at (-0.5, 0.5) {};
		\node [style=none] (6) at (0, 0.25) {$\bar\beta$};
		\node [style=green vertex] (7) at (-0.5, -0.5) {};
		\node [style=none] (9) at (0.25, -0.75) {};
		\node [style=none] (10) at (-2, -0.25) {=};
		\node [style=cointegral] (13) at (-0.5, -0.75) {};
		\node [style=none] (14) at (-3, 0) {};
		\node [style=none] (15) at (-3.5, 0) {};
		\node [style=none] (16) at (-3.75, 0) {};
		\node [style=none] (17) at (-2.75, 0) {};
		\node [style=none] (18) at (-2.75, 0.5) {};
		\node [style=none] (19) at (-3.75, 0.5) {};
		\node [style=none] (20) at (-3.25, 0.25) {$\bar\beta$};
		\node [style=none] (21) at (-3, -0.75) {};
		\node [style=none] (22) at (-3.5, -0.75) {};
		\node [style=red vertex] (23) at (-1.25, 0.5) {};
		\node [style=antipode] (24) at (-1.5, -0.25) {};
		\node [style=none] (25) at (-1.5, -0.75) {};
		\node [style=integral] (26) at (-1.25, 0.75) {};
		\node [style=none] (27) at (1, -0.25) {=};
		\node [style=red vertex] (28) at (2.25, 0.5) {};
		\node [style=antipode] (29) at (2, -0.25) {};
		\node [style=none] (30) at (2, -0.75) {};
		\node [style=integral] (31) at (2.25, 0.75) {};
		\node [style=none] (32) at (2.5, -0.75) {};
	\end{pgfonlayer}
	\begin{pgfonlayer}{edgelayer}
		\draw (3.center) to (2.center);
		\draw (4.center) to (3.center);
		\draw (5.center) to (4.center);
		\draw (2.center) to (5.center);
		\draw [style=directed edge, in=45, out=-90] (1.center) to (7);
		\draw [style=directed edge] (0.center) to (9.center);
		\draw [style=directed edge] (7) to (13);
		\draw (17.center) to (16.center);
		\draw (18.center) to (17.center);
		\draw (19.center) to (18.center);
		\draw (16.center) to (19.center);
		\draw [style=directed edge] (14.center) to (21.center);
		\draw [style=directed edge] (15.center) to (22.center);
		\draw [bend right=15] (23) to (24);
		\draw (25.center) to (24);
		\draw (26) to (23);
		\draw [style=directed edge, in=135, out=-45] (23) to (7);
		\draw [bend right=15] (28) to (29);
		\draw (30.center) to (29);
		\draw (31) to (28);
		\draw [style=directed edge, bend left=15] (28) to (32.center);
	\end{pgfonlayer}
\end{tikzpicture}
  \]
  hence, $\gamma$ is the right inverse of $\beta$, and $H$ is
  nondegenerate. The proof for $\gamma'$ is similar.
\end{proof}

\begin{lemma}\label{lem:I-ind-of-half-dual}
  When $H$ has two half dual structures,
  $\hopfisfrob{triangle}, \hopfisfrob{cotriangle}$ and
  $\hopfisfrob{square-cap}, \hopfisfrob{square-cup}$, then the
  integral morphisms coincide.
\end{lemma}
  
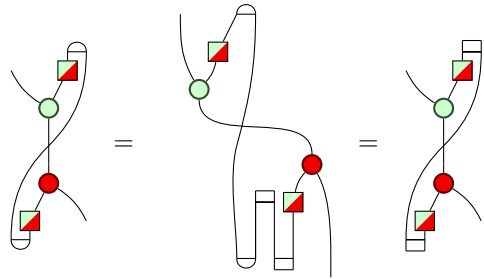
\begin{proof}
    \[
    \hopfisfrob{Integral-sections-coincide}
    \]
  \end{proof}

\begin{prevlemma}{lem:I-preserves-integrals}
  Given a point $p:I \rightarrow H$, and copoint $q:H \rightarrow I$,
  the morphism $\Isec \circ p$ is a left cointegral, and
  $q \circ \Isec$ is a right integral. In addition, $p$ is a left
  cointegral if and only if $\Isec \circ p = p$, and $q$ is a right
  integral if and only if $q \circ \Isec = q$.
%% \[
%% \hopfisfrob{projection-acts-as-integral}
%% \]
\end{prevlemma}

\begin{proof}
  Out goal is to show that, for all points $p$,
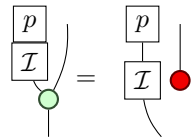
  \[\begin{tikzpicture}
	\begin{pgfonlayer}{nodelayer}
		\node [style=green vertex] (0) at (-0.5, -0.25) {};
		\node [style=scalar box] (1) at (-0.75, 0.25) {\Isec};
		\node [style=none] (3) at (-0.25, 0.75) {};
		\node [style=none] (4) at (-0.5, -0.75) {};
		\node [style=none] (5) at (0, 0) {=};
		\node [style=scalar box] (6) at (0.75, 0) {\Isec};
		\node [style=none] (8) at (1.25, 0.75) {};
		\node [style=none] (9) at (1, -0.75) {};
		\node [style=red vertex] (10) at (1.25, 0) {};
		\node [style=scalar box] (11) at (-0.75, 0.75) {$p$};
		\node [style=scalar box] (12) at (0.75, 0.75) {$p$};
	\end{pgfonlayer}
	\begin{pgfonlayer}{edgelayer}
		\draw [bend left=15] (3.center) to (0);
		\draw [bend right=15] (1) to (0);
		\draw (0) to (4.center);
		\draw [bend right=15] (6) to (9.center);
		\draw (8.center) to (10);
		\draw [style=directed edge] (12) to (6);
		\draw [style=directed edge] (11) to (1);
	\end{pgfonlayer}
\end{tikzpicture}\]
  If we are able to prove the following, then the result will follow.
  
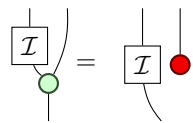
\[\begin{tikzpicture}
	\begin{pgfonlayer}{nodelayer}
		\node [style=green vertex] (0) at (0, -0.25) {};
		\node [style=scalar box] (1) at (-0.25, 0.25) {\Isec};
		\node [style=none] (2) at (-0.25, 0.75) {};
		\node [style=none] (3) at (0.25, 0.75) {};
		\node [style=none] (4) at (0, -0.75) {};
		\node [style=none] (5) at (0.5, 0) {=};
		\node [style=scalar box] (7) at (1.25, 0) {\Isec};
		\node [style=none] (8) at (1.25, 0.75) {};
		\node [style=none] (9) at (1.75, 0.75) {};
		\node [style=none] (10) at (1.5, -0.75) {};
		\node [style=red vertex] (11) at (1.75, 0) {};
	\end{pgfonlayer}
	\begin{pgfonlayer}{edgelayer}
		\draw [bend left=15] (3.center) to (0);
		\draw [bend right=15] (1) to (0);
		\draw (2.center) to (1);
		\draw (0) to (4.center);
		\draw (8.center) to (7);
		\draw [bend right=15] (7) to (10.center);
		\draw (9.center) to (11);
	\end{pgfonlayer}
\end{tikzpicture}\]
  As such, we may begin by composing \Isec with \gmult.
  
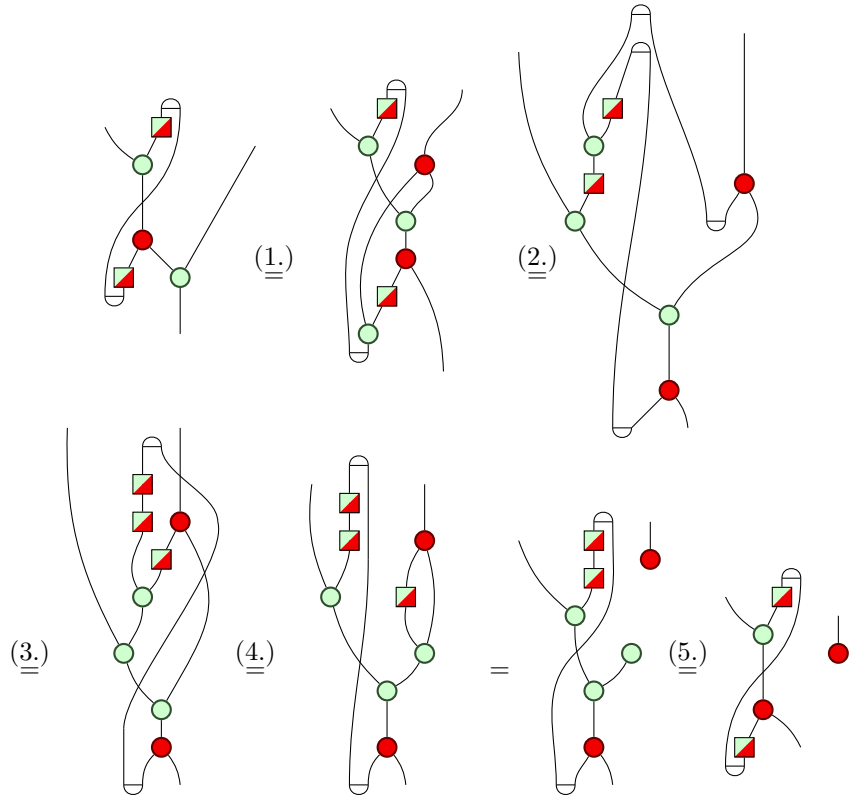
  \[\input{hopf-is-frob/half-dual-section-proof-1.tikz}\]
  
  Where $(1.)$ comes from Lemma \ref{lemma:alternate-bialgebra}, (2.)
  comes from the presence of half duals. The antipode is a bialgebra
  homomorphism $\Hop \rightarrow H$, by Proposition
  \ref{prop:antipode-props}, which gives us (3.). Associativity gives
  us (4.), and (5.) is due to the presence of half duals. Hence, we
  have proved our result.

  Out result also tells us that if $\Isec \circ p = p$, then $p$ is a
  left cointegral. For the converse, let
  \rint be a left cointegral. We then get
  \[
  \hopfisfrob{projection-identity-on-integrals}
  \]
  The proof for right integrals is similar.
\end{proof}

\begin{prevlemma}{lem:equaliser}
    Let the object $H$ have a right half dual $H^*$, where $H$ is a
    Hopf algebra. $H$ fulfills the Frobenius condition if and only if
    there is an equaliser \rint of
    \[
    \inlfropf{equaliser-statement}
    \]
    if and only if there is a coequaliser \gcoint of
    \[
    \inlfropf{equaliser-statement-2}
    \]
\end{prevlemma}

\begin{proof}
  Suppose that $H$ fulfills the Frobenius condition. Then
  $(H,\rint, \gcoint)$ is an integral Hopf algebra
  by Theorem \ref{theorem:frob}, so
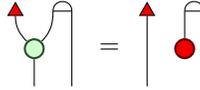
  \[
    \inlfropf{equaliser-proof-1}
  \]
  We shall actually prove that \rint is a split equaliser. To do so,
  we need to find a retract of \rint and
  \inlfropf{equaliser-proof-4}. The Frobenius condition tells us that
  \gcoint is the retract of \rint, and we can easily calculate that
  the morphism \inlfropf{equaliser-proof-2} is a retract of
  \inlfropf{equaliser-proof-4}. To show that \rint must be a split
  equaliser, all we need to show now is
  \[
    \inlfropf{equaliser-proof-3}
  \]
  but this follows immediately from the assumption that the Frobenius
  condition is satisfied. Thus, we have one direction. Showing that
  \gcoint is a split coequaliser follows dually.

  For the other direction, note that by Lemma
  \ref{lem:I-preserves-integrals} we have  
    \[
    \inlfropf{equaliser-proof-5}
    \]
  where \Isec is the integral morphism. Thus, \Isec is a cone of the
  appropriate diagram. We are assuming that \rint is
  an equaliser, so there is a unique
  morphism $\gcoint:H \rightarrow I$ such that
    \[
    \inlfropf{equaliser-proof-3}
    \]
    Also, since \rint is a cointegral, by
    Lemma~\ref{lem:I-preserves-integrals} we get that
    $\rint = \mathcal{I}\circ \rint = \rint \circ \gcoint \circ
    \rint$. Since \rint is an equaliser, $\gcoint \circ \rint =
    1_I$. Hence, the Frobenius condition is satisfied. It is clear
    that if we assume that we have a coequaliser, \gcoint, this will
    also imply the Frobenius condition by duality.

\end{proof}

\begin{prevlemma}{lem:integrals-are-multiples}
  Let $(H,\rint,\gcoint)$ be an integral Hopf algebra and suppose that
  $H$ has enough points. If every left cointegral (right integral) is
  a scalar multiple of \runit (resp. \gcounit) then $H$ fulfills the
  Frobenius condition
\end{prevlemma}
\begin{proof}
  % Suppose that $H$ fulfills the Frobenius condition. Then Lemma
  % \ref{lem:equaliser} imples that \runit is an equaliser, and left
  % cointegrals are, by definition, cones of the appropriate
  % diagram. Hence, it follows that every left cointegral is a unique
  % scalar multiple of \runit. The same follows for \gcounit.

  % Suppose the converse, and that $H$ has enough points. In addition,
  % suppose that there exists an integral and cointegral pair, \gcoint
  % and \rint, such that every left cointegal is a unique scalar
  % multiple of \rint.
  If we can show that
  \[
    \hopfisfrob{scalar-multiple-proof-4}
  \]
  then, since $H$ has enough points, the result will follow. By Lemma
  \ref{lem:I-preserves-integrals}, for all points $a$, $\Isec \circ a$
  is a cointegral. Then, by hypothesis there exists a scalar
  $k:I \rightarrow I$ such that $\Isec \circ a = \rint \OX k$
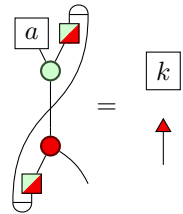
  \[
    \hopfisfrob{scalar-multiple-proof-2}
  \]
  Hence, if we can show that $\gcoint \circ a = k$, we will have our
  result. Observe that, since \gcoint is an integral,
  $\gcoint\circ\Isec=\gcoint$, so we get the following.
  \[
    \hopfisfrob{scalar-multiple-proof-3}
  \]
  and the result follows.
  % This also implies that each scalar multiple
  % is unique. We shall also need to show that
  % $\gcoint \circ \rint = 1_I$. We apply the previous result to get
  % \[
  %   \hopfisfrob{scalar-multiple-proof-5}
  % \]
  % Since every scalar multiple is a unique scalar multiple of \rint, it
  % must be the case that $\gcoint \circ \rint = 1_I$.
  
\end{proof}

%% \begin{prevlemma}{lem:integrals-are-multiples}(forward direction)
%%   Let $H$ be a Hopf algebra that fulfills the Frobenius
%%   condition. Every left cointegral (right integral) is a scalar
%%   multiple of \runit (resp. \gcounit).
%% \end{prevlemma}
%% \begin{proof}
%%   Let \rint be a left cointegral on $H$. Then by Lemma
%%   \ref{lem:I-ind-of-half-dual}, as $\rint = \rint \circ \Isec$
%%   \[
%%   \hopfisfrob{integral-scalar-multiple}
%%   \]
%%   The proof for right integrals is similar.
%% \end{proof}

%% \begin{prevlemma}{lem:integrals-are-multiples}(backward direction)
%%   Let $H$ have enough points, and $(H,\rint,\gcoint)$ be an integral
%%   Hopf algebra. Suppose that every left cointegral (right integral) is
%%   a scalar multiple of \rint (resp. \gcoint). Then $H$ fulfills the
%%   Frobenius condition.
%% \end{prevlemma}
%% \begin{proof}
%%   By Lemma \ref{lem:I-preserves-integrals}, for all points $a$
%%   (copoint $b$), $\Isec \circ a$ is a cointegral and $b \circ\Isec$ is
%%   an integral. Then, by hypothesis

%% \[ \hopfisfrob{scalar-multiple-proof-2} \]

%%   for some scalar $k$. Since \gcoint is an integral,
%%   $\gcoint\circ\Isec=\gcoint$, so we get the following

%% \[ \hopfisfrob{scalar-multiple-proof-3} \]

%%   Hence, we see that

%% \[ \hopfisfrob{scalar-multiple-proof-4} \]

%%    Since $H$ has enough points, we observe that the Frobenius
%%    condition is satisfied.
%% \end{proof}

\begin{prevtheorem}{thm:iff-frob-condition}
  Let $H$ be a Hopf algebra such that the object $H$ has some weak
  right dual $H^*$. Then $H$ admits a Hopf-Frobenius algebra structure if and only
  if $H$ fulfills the Frobenius condition.
\end{prevtheorem}
\begin{proof}
  If $H$ is a Hopf-Frobenius algebra, then it admits a Frobenius
  algebra, and therefore, by Theorem \ref{theorem:admits-Frobenius}, it
  fulfills the Frobenius condition.

  Consider the converse. In what follows, we will prove the Theorem by
  first proving some intermediary lemmas. If $H$ fulfills the
  Frobenius condition, then this is equivalent, by Theorem
  \ref{theorem:frob}, to $H$ admitting a pre-HF structure such that
  $(H,\runit,\gcounit)$ is an integral Hopf algebra. We begin by
  proving that $\antipode=\antipodeform$.

% \begin{prevlemma}{lem:new-hopf-algebra}
%   Let $H$ admit a pre-HF algebra structure; $(H,\runit,\gcounit)$ is
%   an integral Hopf algebra if and only if
%   $\antipode=\antipodeform$. Under these conditions, we have the
%   following facts
%   \begin{itemize}
%   \item let $(\cdot)\gtrans$ be the duality defined by the green
%     Frobenius algebra (cf.\!  Lemma~\ref{lem:frob-self-dual}).  Then:
%   \suck
%   \[
%   \left( \inltf{green-mult}\!\!\right)\gtrans =
%   \inltf{green-comult}
%   \qquad\qquad
%   \left( \inltf{red-comult}\!\!\right)\gtrans =
%   \inltf{red-mult-swap}
%   \]
%   \suck
%   \item $(H,\rmult,\runit,\gcomult,\gcounit,\antipoded)$ forms a Hopf
%   algebra, where $\antipoded =\daggerantipode$.
      
%   \end{itemize}
% \end{prevlemma}
% \begin{proof}
%   To prove this, we will
% \end{proof}

\begin{lemma}\label{lemma:iff-integral}
  Let $H$ admit a pre-HF algebra structure; $(H,\runit,\gcounit)$ is an integral
  Hopf algebra if and only if $\antipode=\antipodeform$.
\end{lemma}
\begin{proof}
  % Note first that $\antipode=\antipodeform$ if and only if $\antipode$
  % has an inverse, $\antipode^{-1} = \antipodeiform$.  We will use this
  % second form in the following proof.

  The implication in one direction follows from
  Lemma~\ref{theorem:invertible-antipode}.
  Suppose the converse. Note that $\antipode = \antipodeform$ is equivalent
  to $\antipode^{-1} = \antipodeiform$. We use the fact that the
  antipode is a bialgebra homomorphism to get the following.
  
  \[ \hopfisfrob{antipode-inverse-proof-2} \]
  
  It follows from this that that we may express $\runit$ similarly.
  
  \[ \hopfisfrob{antipode-inverse-proof-3} \]
  
  This allows us to show that $\runit$ is a left cointegral.
  
  \[ \hopfisfrob{antipode-inverse-proof-4} \]
  
  The proof that $\gcounit$ is a right integral is similar. We only
  need to show that $\gcounit \circ \runit = 1_I$, but this follows
  from above.
  \[ \hopfisfrob{antipode-inverse-proof-4b} \]
  
\end{proof}

It follows immediately from this Lemma that when $H$ fulfills the
Frobenius condition, the antipode is the canonical isomorphism that
maps from one dual structure to the other, in the sense of Proposition
\ref{prop:duals-are-unique}. We record this fact as a Corollary.
\begin{corollary}\label{cor:green-to-red-cup-1}
  Let $H$ admit a pre-HF algebra structure, such that $(H,\runit,\gcounit)$ is an
  integral Hopf algebra. Then
  \[ \hopfisfrob{green-to-red-cup-1} \]
\end{corollary}
% \begin{proof}
%   This comes straight from the fact that $\antipode=\antipodeform$ and
%   $\antipode^{-1} = \antipodeiform$
% \end{proof}

To prove the Theorem, we must show that
$(H,\rmult,\runit,\gcomult,\gcounit,\antipoded)$ forms a Hopf algebra,
where $\antipoded =\daggerantipode$. We will accomplish this by showing
first that $(H,\rmult,\runit,\gcomult,\gcounit)$ forms a bialgebra,
and then that $\antipoded$ is the appropriate antipode. Recall that
the dual of a Hopf algebra is a Hopf algebra, in the sense
of Definition \ref{def:dual-hopf-alg}. By using the dual structure of
the green Frobenius algebra, we get the following:

\begin{lemma}\label{lemma:how-green-relates-to-dagger}
  Let $H$ admit a pre-HF algebra structure such that $(H,\runit,\gcounit)$ is
  an integral Hopf algebra, and let $(\cdot)\gtrans$ be the duality
  defined by the green Frobenius algebra (cf.\!
  Lemma~\ref{lem:frob-self-dual}).  Then:
  \suck
  \[
  \left( \inltf{green-mult}\!\!\right)\gtrans =
  \inltf{green-comult}
  \qquad\qquad
  \left( \inltf{green-unit}\!\!\right)\gtrans =
  \inltf{green-counit}
  \qquad\qquad
  \left( \inltf{red-comult}\!\!\right)\gtrans =
  \inltf{red-mult-swap}
  \qquad\qquad
  \left( \inltf{red-counit}\!\!\right)\gtrans =
  \inltf{red-unit}
  \]
  \suck
\end{lemma}

\begin{proof}
  The first two statements are clear from the definition of the green
  dual. For the third statement, we see that
  \[
  \hopfisfrob{green-functor-on-red}
  \]
  where (1.) comes from Corollary \ref{cor:green-to-red-cup-1}. The
  final statement follows from above, as $(\rcounit)\gtrans$ will
  be the unit of $(\rcomult)\gtrans$. Units of monoids are unique, so
  $(\rcounit)\gtrans = \runit$.
\end{proof}

We now have that
$H^{\gtrans}:=(H,\rcomult\!{}\gtrans,\rcounit{}\gtrans,\gmult\!{}\gtrans,\gunit{}\gtrans,\antipode{}\gtrans)$
is a Hopf algebra. By the above Lemma,
$(H,\rmult,\runit,\gcomult,\gcounit)$ is simply $(H^{\gtrans})^\sigma$
when viewed as a bialgebra. Hence, by Proposition
\ref{prop:antipode-props}, to show that this is a Hopf algebra, we
only need to show that $\antipode{}\gtrans$ is invertible, and that it
is equal to $\daggerantipode$. But $(\cdot)\gtrans$ preserves
inverses, so we know that $\antipoded = (\antipode^{\gtrans})^{-1}$.
All that remains is showing that $\antipoded$ has the appropriate
form, and this is done by straightforward calculation.
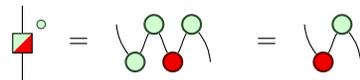
\[
  \hopfisfrob{green-dual-antipode}
\]
Hence, $\antipoded
=(\antipode^{\gtrans})^{-1}=\daggerantipode$. Therefore, if $H$
fulfills the Frobenius condition, the $H$ admits a Hopf-Frobenius
algebra structure.

\end{proof}

\begin{corollary}\label{cor:green-to-red-cup-2}
  Let $H$ admit a pre-HF algebra structure, such that $(H,\runit,\gcounit)$ is an
  integral Hopf algebra. Then
  \[ \hopfisfrob{green-to-red-cup-2} \]
\end{corollary}

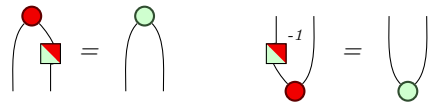
\begin{prevlemma}{lem:HF-is-unique}
    Let $H$ admit a Hopf-Frobenius algebra structure. Then this
    structure is unique up to invertible scalar.
\end{prevlemma}
\begin{proof}
  Suppose that $(H,\gmult,\gunit,\rcomult,\rcounit,\antipode)$ admits
  two Hopf-Frobenius structures,
  $(H,\rmult,\runit,\gcomult,\gcounit,\antipoded)$ and
  $(H,\altrmult,\altrunit,\altgcomult,\altgcounit,\overline\antipoded)$.
  Recall that we refer to
  $(H,\gmult,\gunit,\rcomult,\rcounit,\antipode)$ as the green Hopf
  algebra. The Hopf-Frobenius structures share a green Hopf algebra,
  so the respective units and counits of these Hopf algebras must be
  left cointegrals and right integrals of the green Hopf algebra. It
  follows from Lemma \ref{lem:equaliser} that there exists unique
  scalars, $k,l:I \rightarrow I$, such that $\runit = \altrunit \OX k$
  and $\gcounit = \altgcounit \OX l$. Since these are both Hopf
  algebras, we get the following
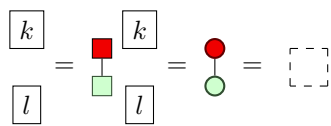
\[\hopfisfrob{HF-is-unique-proof-1} \]
Thus, $k$ and $l$ are mutually inverse.
We now only need to show that the other structure maps are scalar
multiples of each other. We see that $\overline\antipoded =
\antipoded$ follows from the above proof, as
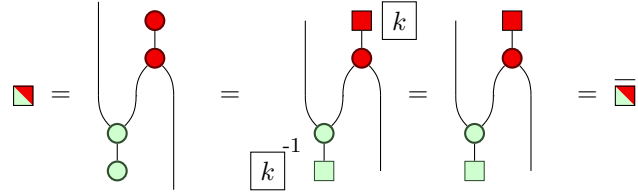
\[\hopfisfrob{HF-is-unique-proof-3} \]
Finally, this will imply that the multiplication and comultiplication
maps only differ by an invertible scalar. Note that as $\overline\antipoded =
\antipoded$, their inverses will also coincide. Recall how \altgcomult
is constructed, and observe that
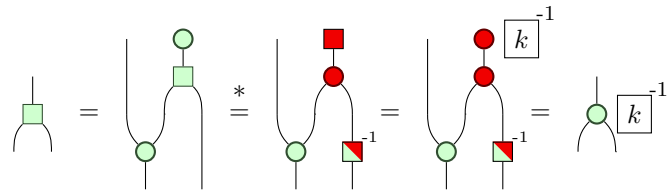
\[\hopfisfrob{HF-is-unique-proof-2} \]
where (*) comes from Corollary \ref{cor:green-to-red-cup-2}. The same
is true for \altrmult and \rmult. Hence, if $H$ admits two
Hopf-Frobenius algebra structures, then they will only differ by an
invertible scalar factor.

\end{proof}

\begin{prevlemma}{lemma:rediso}
   The morphism \rediso is a Hopf algebra homomorphism between
  $H_{\rconj}^{\sigma}$ and $H_{\gconj}^*$.
\end{prevlemma}
\begin{proof}
We will only show that \rediso is a homomorphism for $\gcomult\!{}^*$,
the rest of the structure maps will have similar proofs. We first note
that, by Corollary \ref{cor:green-to-red-cup-1}
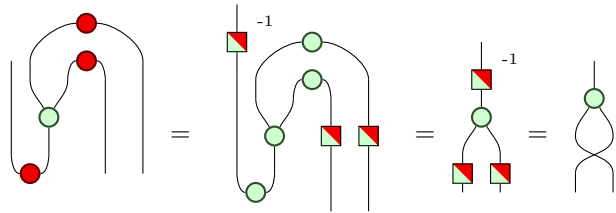
\[
\inlfropf{rediso-proof-2}
\]
Hence, we see that
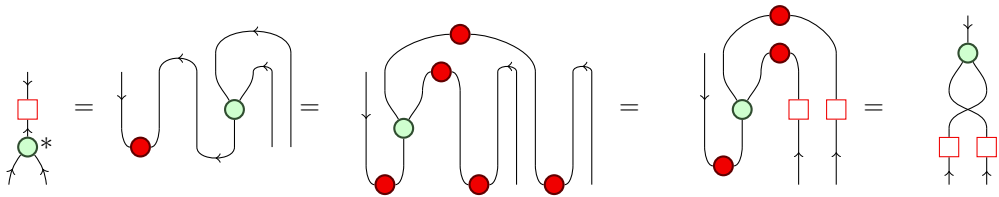
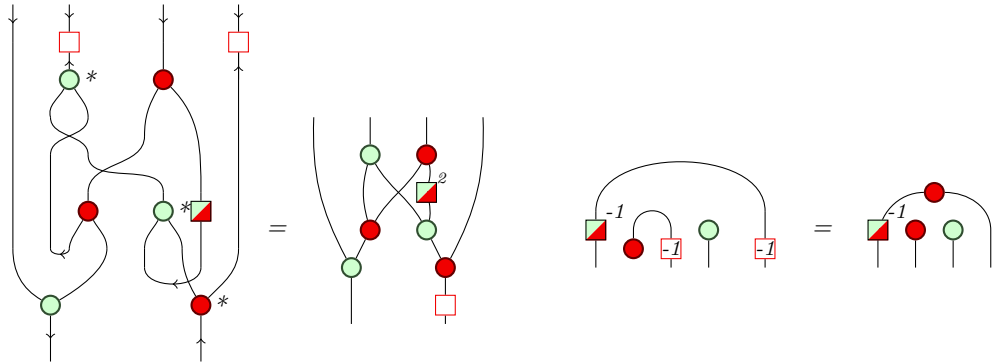
\[
\inlfropf{rediso-proof-1}
\]
\end{proof}

\begin{prevlemma}{lemma:rho-with-D-double}
  \[
  \hopfisfrob{rho-with-D-double}
  \]
\end{prevlemma}
\begin{proof}
  This is clear from the definition of \rediso, Lemma \ref{lemma:rediso}
  and Corollary \ref{cor:green-to-red-cup-1}. We explicitly spell out
  the first statement here.
  \[
  \hopfisfrob{rho-with-D-double-proof-1}
  \]
  The proof of the second statement follows immediately from the
  definition of \redinv.
\end{proof}

\section{Taft Hopf algebra for $n=2$}

Here we shall state the Hopf-Frobenius algebra for the 4 dimensional
Taft Hopf algebra explicitly. It is generated by $g$ and $x$, and has
the structure

\begin{table}[h]
\begin{tabular}{l|llllll|lll|lll|l}
$\gmult$ & $1$  & $x$  & $g$   & $gx$ &  & $\rcomult$ &                      &  & $\rcounit$ &     &  & $\antipode$ &      \\ \cline{1-5} \cline{7-8} \cline{10-11} \cline{13-14} 
$1$      & $1$  & $x$  & $g$   & $gx$ &  & $1$        & $1 \otimes 1$        &  & $1$        & $1$ &  & $1$         & $1$  \\
$x$      & $x$  & $0$  & $-gx$ & $0$  &  & $x$        & $1 \OX x + x \OX g$  &  & $x$        & $0$ &  & $x$         & $gx$ \\
$g$      & $g$  & $gx$ & $1$   & $x$  &  & $g$        & $g \OX g$            &  & $g$        & $1$ &  & $g$         & $g$  \\
$gx$     & $gx$ & $0$  & $-x$  & $0$  &  & $gx$       & $g \OX gx+ gx \OX 1$ &  & $gx$       & $0$ &  & $gx$        & $-x$
\end{tabular}
\end{table}

\FVect is a compact closed category, so the integral projection is the
map
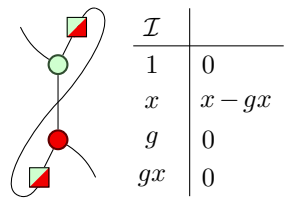
\[
\hopfisfrob{sweedler-algebra-Isec}
\]
Hence, the element $x-gx$ is a left cointegral, and the right integral
is the delta function for $x$, $\delta_x$. Hence, by Theorem
\ref{theorem:frob}, these shall be our unit and counit
respectively. It is now possible to construct the resulting
Hopf-Frobenius algebra, but we shall explicitly state the structure
maps. The green Frobenius algebra is
\[
\hopfisfrob{Sweedler-algebra-structure-1}
\]
and the red Frobenius algebra
\[
\hopfisfrob{Sweedler-algebra-structure-2}
\]

\section{Additional Background Material}
\label{sec:addit-backgr-mater}

In this section we provide additional definitions and basic properties
to flesh out the background material of Section~\ref{sec:preliminaries}.

\subsection{Categories with duals}
\label{sec:more-categ-with-duals}

\begin{proposition}\label{prop:duals-are-unique}
  In a monoidal category \catC suppose that $A$ has two right duals
  $(B_1,d_1,e_1)$ and $(B_2,d_2,e_2)$; then there exists an
  isomorphism $f: B_1 \isomorphism B_2$, satisfying the equations
  shown below.
  \[
  \inltf{f-map} \! := \ \inltf{f-defn} 
  \qquad \qquad
  \inltf{duals-are-uniq-2-lhs} = 
  \inltf{duals-are-uniq-2-rhs}  
  \qquad \qquad
  \inltf{duals-are-uniq-1-lhs} = 
  \inltf{duals-are-uniq-1-rhs}  
  \]
  \begin{proof}
    Define $f$ as shown above; the required equations follow immediately.
  \end{proof}
\end{proposition}

\subsection{Monoids and Comonoids}
\label{sec:monoids-comonoids}

\begin{definition}\label{def:monoid-comonoid}
  A \emph{monoid} in a monoidal category \catC consists of an object $M$,
  a binary multiplication $\mu : M \otimes M \to M$ and a unit
  morphism $\eta : I \to M$ obeying the familiar associativity and
  unit laws, shown in diagram form below.
  \[
  \inltf{green-assoc-lhs} = \inltf{green-assoc-rhs}
  \qquad\qquad
  \inltf{green-unit-left} = \inltf{identity_1} =
  \inltf{green-unit-right} 
  \]
  A \emph{comonoid} in \catC is a monoid in $\catC\op$, concretely
  depicted below.
  \[
  \inltf{green-coassoc-lhs} = \inltf{green-coassoc-rhs}
  \qquad\qquad
  \inltf{green-counit-left} = \inltf{identity_1} =
  \inltf{green-counit-right} 
  \]
  A (co)monoid is called (co)commutative if its (co)multiplication is
  invariant under the exchange map, as depicted below.
  \[
  \inltf{green-swap-mult} =   \inltf{green-mult}
  \qquad\qquad
  \inltf{green-swap-comult} =   \inltf{green-comult}
  \]
\end{definition}

In this paper we will \emph{not} assume commutativity or
cocommutativity.  

\begin{definition}\label{def:invertible-element}
  Given a monoid $(M,\gmult,\gunit)$, a point $a : I \to M$ is
  \emph{left invertible} if there exists a point $l:I\to M$ satisfying
  the left equation below; it is \emph{right invertible} if there
  exists $r:I\to M$ satisfying the right equation; it is
  \emph{invertible} if it is both left and right invertible, in which
  case the two inverses coincide.
  \[
  \inltf{green-left-inverse} = 
  \inltf{green-unit} =
  \inltf{green-right-inverse}
  \]
  Co-invertibility of co-points $\alpha:M\to I$ with respect to a
  comonoid is defined dually.
\end{definition}

\REM{Add here comment about point copoint element terminology}

\subsection{Frobenius algebras}
\label{sec:more-frobenius-algebras}

In the following lemmas we will assume that
we have a given Frobenius algebra
$(F,\gmult,\gunit,\gcomult,\gcounit)$.

\begin{definition}\label{def:symmetric-frob}
  A Frobenius algebra is called \emph{symmetric} if its cap (or
  equivalently its cup) is invariant under the symmetry.
  \[
  \inltf{green-cap-swap} = \inltf{green-cap} \qquad\qquad\qquad
  \inltf{green-cup-swap} = \inltf{green-cup}
  \]
\end{definition}
\begin{proof}
  See Kock \cite{Kock:TQFT:2003}.
\end{proof}

\begin{lemma}\label{lem:invertible-coinvertible-bij}
  There is a bijective correspondence between invertible points for
  the monoid and coinvertible copoints for the comonoid.
  \begin{proof}
    Let $(\cdot)\gtrans$ be the duality induced by the cup and cap;
    then $u:I\to F$ is invertible iff and only if $u\gtrans : F \to I$
    is coinvertible.
  \end{proof}
\end{lemma}
%\NOTE{Should we reference Kock here? He has this as a lemma, but only for Vect}

\begin{lemma}\label{lem:coinvertible-elts-give-frobs}
  Let $u$ be a coinvertible element of the comonoid.  Define 
  \[
  \beta(u) := \inlfropf{other-frob-structure}
  \qquad  \qquad  \qquad
  \bar\beta(u) := \inlfropf{other-frob-cap}
  \]
  Then $\beta(u)$ is a Frobenius form for the monoid $(F,\gmult,\gunit,)$.
\begin{proof}
  We must show that the equations of \ref{def:frob-alt} hold.  The
  first follows from associativity of the monoid.  For the second we
  have:
  \[
  \inltf{alt-frob-snake-left} =
  \inltf{alt-frob-snake-left-ii} =
  \inltf{alt-frob-snake-left-iii} =
  \inltf{alt-frob-snake-left-iv} = \ \
  \inltf{identity_1}
  \]
  and similarly for the other side.  Note that $\beta(u) = \beta(v)$
  implies $u = v$ by the  uniqueness of inverses.
\end{proof} 
\end{lemma}
\begin{lemma}\label{lem:frobs-give-coinvertible-elts}
  Suppose that $\beta$ is a Frobenius form on
  \gmult\!;  then we obtain a coinvertible element $u:F\to I$ as follows:
  \[
  u := \inltf{alt-frob-coinvertible}
  \qquad  \qquad  \qquad
  u^{-1} := \inltf{alt-frob-coinverse}
  \]
  \begin{proof}
    We need only to show that $u^{-1}$ is the coinverse of $u$.
    \[
    \inltf{alt-frob-coinverse-pf-i} \ = \ 
    \inltf{alt-frob-coinverse-pf-ii} \ = \ 
    \inltf{alt-frob-coinverse-pf-iii} \ = \ 
    \inltf{alt-frob-coinverse-pf-iv} \ = \ 
    \inltf{alt-frob-coinverse-pf-v} \ = \ \ 
    \inltf{green-counit}
    \]
  \end{proof}
\end{lemma}

\noindent
Combining the three preceding lemmas we obtain:

\begin{proposition}
  \label{prop:frob-structs-are-invertible-elts}
  There is a bijective correspondence between the invertible elements
  of a monoid and the Frobenius forms definable on it.
\end{proposition}

  % \ROUGH{
  % Suppose we have another Frobenius structure on $\mu$. Then there is
  % some morphism $v:F \rightarrow I$ such that we may find a cap where
  % we have the relation
  % \[
  % \hopfisfrob{new-frob-identity}
  % \]
  % This will give us the following
  % \[
  % \hopfisfrob{invertible-counit}
  % \]}

%%% Local Variables: 
%%% mode: latex
%%% TeX-master: t
%%% End: 

\end{document}